\documentclass[12pt]{amsart}
\usepackage{fullpage}

\usepackage{amsmath,amsthm,amsfonts,amscd, amssymb}
\usepackage[alphabetic]{amsrefs}

\usepackage{tikz, tikz-cd}
\usetikzlibrary{matrix,arrows,decorations.pathmorphing}

\newcommand{\po}{\arrow[dr, phantom, "\ulcorner" near start]}

\newcommand{\oo}{\infty}
\newcommand{\cat}{\mathcal}
\newcommand{\colim}{\operatorname{colim}}

\let\term=\emph

\newcommand{\ac}{\textrm{ac}}
\newcommand{\qi}{\textrm{qi}}
\newcommand{\chhtpy}{\textrm{ch}}
\newcommand{\coker}{\operatorname{coker}}
\newcommand{\fg}{\textrm{fg}}

\newcommand{\heart}{\heartsuit}
\newcommand{\Aut}{\operatorname{Aut}}
\newcommand{\GL}{\operatorname{GL}}
\newcommand{\coloneq}{\mathrel{\mathop:}=}
\newcommand{\Hom}{\operatorname{Hom}}
\newcommand{\eg}{{\it e.g.,\ }}
\newcommand{\Sp}{\operatorname{Sp}}
\newcommand{\HZ}{H\mathbb{Z}}
\newcommand{\Z}{{\mathbb Z}}
\newcommand{\comp}{\mathbin{\circ}}
\newcommand{\id}{\operatorname{id}}
\newcommand{\cofiber}{\operatorname{cofiber}}
\newcommand{\fiber}{\operatorname{fiber}}
\newcommand{\twiddle}{\widetilde}

\newcommand{\Ch}{\operatorname{Ch}}
\newcommand{\Proj}{\operatorname{Proj}}
\newcommand{\Perf}{\operatorname{Perf}}
\newcommand{\im}{\operatorname{im}}
\newcommand{\fin}{\textrm{fin}}
\newcommand{\ie}{{\it i.e.,\ }}
\newcommand{\cf}{{\it cf.\ }}

\newcommand{\Mod}{\operatorname{Mod}}
\newcommand{\bdd}{\textrm{bdd}}
\newcommand{\triv}{\textrm{triv}}

\newcommand{\cell}{\operatorname{Cell}}
\newcommand{\Fun}{\textrm{Fun}}
\newcommand{\Wald}{\textrm{Wald}}
\newcommand{\tr}{\operatorname{tr}}
\newcommand{\cotr}{\operatorname{cotr}}
\newcommand{\cof}{\rightarrowtail}
\newcommand{\wto}{\xrightarrow{\sim}{}}
\newcommand{\x}{\times}
\newcommand{\inv}{\ensuremath{^{-1}}}

\newcommand{\incl}{\operatorname{incl}}

\newcommand{\isom}{\cong}
\newcommand{\Stab}{\operatorname{Stab}}
\newcommand{\cone}{\operatorname{cone}}

\newcommand{\Diff}{\textrm{Diff}}
\newcommand{\loops}{\Omega}

\newcommand{\sh}{\operatorname{sh}}
\newcommand{\hocolim}{\operatorname{hocolim}}

\usepackage{url}

\theoremstyle{plain}
\newtheorem{theorem}{Theorem}[section]
\newtheorem*{theorem*}{Theorem}
\newtheorem{corollary}[theorem]{Corollary}
\newtheorem*{corollary*}{Corollary}
\newtheorem{lemma}[theorem]{Lemma}
\newtheorem{proposition}[theorem]{Proposition}

\theoremstyle{definition}

\newtheorem{claim}[theorem]{Claim}
\newtheorem{definition}[theorem]{Definition}

\theoremstyle{remark}
\newtheorem{example}[theorem]{Example}
\newtheorem{remark}[theorem]{Remark}

\usepackage{hyphenat}
\AtBeginDocument{
   \def\MR#1{}
}

\begin{document}

\title[Weight structures and algebraic $K$-theory]{Weight structures
 and the algebraic $K$-theory of stable $\oo$-categories}
\author{Ernest E.~Fontes}
\email{fontes.17@osu.edu}
\address{Department of Mathematics\\The Ohio State University\\231 W 18th Ave\\Columbus, OH 43210}

\date{\today}


\begin{abstract}
  \indent We introduce the notion of a bounded weight structure on a
  stable $\infty$-category and use this to prove the natural
  generalization of Waldhausen's sphere theorem: We show that the
  algebraic $K$-theory of a stable $\infty$-category with a bounded
  non-degenerate weight structure is equivalent to the algebraic
  $K$-theory of the heart of the weight structure.
\end{abstract}

\maketitle





\section{Introduction}\label{introduction}

Algebraic $K$-theory is an invariant of rings that arose from
considering Euler characteristics. The classical algebraic $K$-groups
$K_0$, $K_1$, and $K_2$ were defined algebraically and were known to
fit into exact sequences like those for cohomology theories. Quillen
introduced a homotopical construction of an algebraic $K$-theory
spectrum whose homotopy groups extended the classical definitions to
define higher algebraic $K$-theory groups $K_i$ for $i\geq 0$
\cite{quillen}.

The applications for algebraic $K$-theory span several fields. The
algebraic $K$-theory of rings of integers in number fields contains
arithmetic information: the class group, the Brauer group, and so on
\cite{MR3076731}.  On the other hand, the Whitehead torsion of a
manifold $M$ is controlled by $K_1(\mathbb{Z}[\pi_1 M])$.  In a vast
generalization of this, Waldhausen showed that the algebraic
$K$-theory of $\Sigma^{\infty} (\Omega M)_+$, the ``spherical group
ring of the loop space of $M$'', contains information about a
stabilization of $B\Diff(M)$~\cite{1126}.

$K$-theory is very difficult to compute. Trace methods compare it to
topological cyclic homology and compute that in favorable
conditions. Direct computations of $K$-theory involve reduction
theorems, results proving that some $K$-theory is equivalent to the
$K$-theory of something simpler. The chief examples of these reduction
theorems are Quillen's devissage theorem \cite{quillen}*{4} and
Waldhausen's fibration \cite{1126}*{1.6.4}, approximation
\cite{1126}*{1.6.7}, and sphere \cite{1126}*{1.7.1} theorems.

Following Quillen and Waldhausen, the modern view is that algebraic
$K$-theory is a functor of modules categories (or their subcategories)
to spectra. Recent work has extended the construction of algebraic
$K$-theory to higher categories that behave like module categories and
produced universal characterizations of the algebraic $K$-theory
functor in this setting \cites{BGT, barwick}. One consequence of this
work is that most of Quillen and Waldhausen's foundational theorems
about the behavior of algebraic $K$-theory have been established in a
very general context. In particular, this framework has permitted new
localization \cites{BM, barwick}, devissage \cite{barwickheart}, and
approximation \cites{barwick, fiore} results.

However, there has been no counterpart to the sphere theorem in the
$K$-theory of higher categories. In Waldhausen's setting, a category
$\cat C$ equipped with cofibrations, weak equivalences, a cylinder
functor, and a well-behaved homology theory is shown to have the same
$K$-theory, after stabilizing, as the stable homology
spheres. Waldhausen's homology theories are required to satisfy a
condition which makes them display all objects of $\cat C$ as cell
objects weakly built out of homology spheres. Essentially, the sphere
theorem reduces the $K$-theory of a category to the $K$-theory of a
subcategory when those objects can construct all the other objects in
the category as finite cell complexes.

To provide  such a theorem for  higher categories, we need  a suitable
notion of cell  object. We will use Bondarko's  theory of \emph{weight
  structures}.

Weight structures were introduced by Bondarko in
\cite{2009arXiv0903.0091B} and studied extensively on triangulated
categories \cites{MR2746283, Bondarko, 2013arXiv1312.7493B, Bondarko,
  2009arXiv0903.0091B}.  A weight structure is a collection of data on
a triangulated category which is designed to provide a weak notion of
cellular filtrations.  Advantageously, weight structures are specified
entirely on the homotopy category of our stable $\oo$-category
$\cat C$. The data of a weight structure consists of choices of
subcategories of objects built with cells in degrees $\leq n$ (or
$\geq n$) so that every object in $\cat C$ has at least one associated
$n$-skeleton for all $n\in \Z$. These skeleta are not assumed to be
functorial in any way and in fact rarely are. We emphasize that a
weight structure should be thought of as providing weak $n$-skeleta
for all objects in $\cat C$.

Familiar examples of weight structures include CW-structures on
spectra and truncation on chain complexes of finitely-generated
modules. The former is bounded on finite spectra and the latter is
bounded when the chain complexes are bounded. Work of Bondarko has
established a weight structure on Voevodsky's category of effective
motives whose heart consists of the Chow motives
\cite{2009arXiv0903.0091B}. On compact objects, this forms a bounded
weight structure as well.

This paper essentially completes the program of lifting fundamental
theorems of Quillen and Waldhausen to the algebraic $K$-theory of
higher categories by providing an analog of Waldhausen's \emph{sphere
  theorem} \cite{1126}*{1.7.1}. Where Waldhausen approaches finite cell
objects with homology functors, we use bounded weight structures.
\begin{theorem*}[Theorem \ref{spherethm}]
  If $\cat C$ is a stable $\oo$-category equipped with a bounded
  non-degenerate weight structure $w$, then the inclusion of the heart
  of the weight structure
  $\cat C_{\heart w}\hookrightarrow \cat C_{\heart}$ induces an
  equivalence on algebraic $K$-theory
  $K(\cat C) \simeq K(\cat C_{\heart w})$.
\end{theorem*}

The main theorem is an example of an equivalence
between $K$-theory spectra that is \emph{not} induced by an
equivalence of derived categories. Quillen's devissage theorem
\cite{quillen}*{4} provides such an equivalence. Blumberg and
Mandell's devissage theorem for ring spectra \cite{BM} and the
closely-related theorem of the heart due to Barwick
\cite{barwickheart} are the only other such non-trivial equivalences
known to the author.

In addition to lifting Waldhausen's sphere theorem to quasicategories,
the main theorem generalizes several previous results
in the literature. Bondarko proves a version of the theorem on $K_0$ for
triangulated categories equipped with bounded weight structures
\cite{MR2746283}*{5.3.1}. We reproduce his result as a corollary of
our theorem.
\begin{corollary*}[Corollary \ref{bondarko_k0}]
  If $\cat T$ is a triangulated category equipped with a
  non-degenerate bounded weight structure $w$, then
  $K_0(\cat T) \simeq K_0(\cat T_{\heart w})$.
\end{corollary*}
The projective weight structure on bounded chain complexes gives an
equivalence of $K(R)$ with the $K$-theory of the category of
finitely-generate projective $R$-modules. Generalizing this result
slightly, we produce a new proofs of the Gillet--Waldhausen Theorem
and the Resolution Theorem for exact categories:
\begin{corollary*}[Corollary \ref{gillet-wald}, Gillet--Waldhausen Theorem, \cite{MR1106918}*{1.11.7}]
  For an exact category $\cat E$ which is idempotent-complete,
  Quillen's algebraic $K$-theory $K(\cat E)$ is homotopy equivalent to
  the Waldhausen algebraic $K$-theory of $\Ch^\bdd(\cat E)$, the
  Waldhausen category of bounded chain complexes on $\cat E$ where
  cofibrations are taken to be admissible monomorphisms and weak
  equivalences are quasi-isomorphisms of chain complexes.
\end{corollary*}
\begin{corollary*}[Corollary \ref{resolution_theorem}, Resolution Theorem, \cite{MR3076731}*{Theorem
    V.3.1}]
  Let $\cat P$ be a full subcategory of an exact category $\cat E$ so
  that $\cat P$ is closed under extensions and under kernels of
  admissible surjections in $\cat E$. Suppose in addition that every
  object $M$ in $\cat E$ admits a finite $\cat P$-resolution:
  \begin{equation*}
    0 \to P_n\to P_{n-1} \to \cdots \to P_1 \to P_0 \to M \to 0
  \end{equation*}
  then $K(\cat P)\simeq K(\cat E)$.
\end{corollary*}

The chief motivation for the main theorem is that $\cat C_{\heart w}$
is simpler than $\cat C$ and so its $K$-theory can be described in an
alternate manner. In particular, all cofiber sequences in the heart
split in the homotopy category.  Hence, $K(\cat C_{\heart w})$ permits
a description in the spirit of Quillen's plus construction for
$K$-theory:
\begin{theorem*}[Corollary \ref{plus-equals-Q}]
  If $\cat C_{\heart w}$ is the heart of a weight structure then all
  cofiber sequences split in the homotopy category and
  \begin{equation*}
    K(\cat C_{\heart w}) \simeq K_0(\cat C) \x (\hocolim_{\cat C_{\heart w, \dagger}} B\Aut(X))^+
  \end{equation*}
  where $[X]$ is an equivalence class of objects in
  $\cat C_{\heart w}$ and $(-)^+$ denotes the group completion of the
  topological monoid.
\end{theorem*}

\subsection{Relation to other results}

It would be appropriate at this point to issue a word of clarification
about the various theorems relating to hearts. Neeman proved a theorem
of the heart for the $K$-theory of triangulated categories
\cites{MR1724625, MR1793672, MR1828612} which was later proven more
generally for exact $\oo$-categories by Barwick
\cite{barwickheart}. In both cases, the theorem said that a bounded
$t$-structure on the homotopy category induced an equivalence $K(\cat
C) \simeq K(\cat C_{\heart t})$ between the $K$-theory of the category
and that of the heart of the $t$-structure. While there are
superficial similarities between the theorems---and, as it turns out,
the definitions of weight and $t$-structures---these theorems have
little to do with one another. The theorems for $t$-structures
generalize Quillen's devissage theorem whereas the theorem for weight
structures generalizes Waldhausen's sphere theorem. A $t$-structure
provides a Postnikov tower for every object of $\cat C$ whereas a
weight structure provides a cellular filtration. Furthermore, while
$t$-structures (and Postnikov towers) are functorial, weight
structures (and cellular filtrations) rarely are.

While this paper was transitioning from dissertation to journal form,
several closely related results have appeared in the
literature.

Sosnilo proves that the algebraic $K$-theory of $\cat C_{\heart w}$,
or $K(\underline{Hw}_\oo)$ in his notation, is homotopy equivalent
$K(\cat C)$ \cite{sosnilo}*{4.1}. Although this result appears
identical to Theorem \ref{spherethm}, his definition of $K(\cat
C_{\heart w})$ differs from ours. The $K$-theory functor in Sosnilo's
work is the nonconnective $K$-theory of stable $\oo$-categories of
\cite{BGT}. Since $\cat C_{\heart w}$ is an additive $\oo$-category,
Sosnilo's $K(\cat C_{\heart w})$ is defined to be the $K$-theory of
the formal closure of $\cat C_{\heart w}$ under finite limits and
colimits \cite{sosnilo}*{p.~15}. If $w$ is bounded, the closure of
$\cat C_{\heart w}$ under finite limits and colimits is canonically
equivalent to $\cat C$.

In contrast, Theorem \ref{spherethm} uses Barwick's $K$-theory functor
for Waldhausen $\oo$-categories without passing through the $K$-theory
of stable $\oo$-categories. {\it A priori,} Barwick's construction for
additive $\oo$-categories need not agree with the $K$-theory of their
formal closure. Theorem \ref{spherethm} proves that it does in the
presence of a bounded weight structure.

Heleodoro appears to have arrived independently at the proof of the
sphere theorem, which appears as \cite{2019arXiv190700384H}*{Theorem
  5} and has a similar proof to the one presented in this
paper. Mochizuki claims the sphere theorem and the theorem of the
heart for $t$-structures as a consequence of a more general result
\cite{2019arXiv190601589M}.

In private correspondence \cite{antieau_email}, Ben Antieau describes
an alternative approach to proving the sphere theorem with the
technology of \cite{EKMM} which may provide an advantageous
perspective for the reader. For any finite set of objects $X$ in $\cat
C_{\heart w}$, let $E(X)$ be the endomorphism spectrum of the finite
wedge of objects in $X$. Since the objects of $X$ lie in the heart of
a weight structure, $E(X)$ is a connective ring spectrum. The additive
closure of $X$ in $\cat C$ is a full subcategory of $\cat C_{\heart
  w}$ and can be identified with $\oo$-category of compact projective
$E(X)$-modules, $\Proj(E(X))$. The full idempotent-complete stable
subcategory of $\cat C$ generated by $X$ can be identified with
$\Perf(E(X))$, the $\oo$-category of compact $E(X)$-modules. We can
now utilize \cite{EKMM}*{IV.7.1} to conclude that $K(\Proj(E(X))) \simeq
K(\Perf(E(X)))$.

Algebraic $K$-theory commutes with filtered colimits, so we take a
colimit over all finite subsets $X$ of the heart. The colimit of the
additive closures will be the heart $\cat C_{\heart w}$, and the
colimit of the stable idempotent-complete closures will be $\cat C$
because the weight structure is bounded. We conclude $K(\cat C_\heart
w) \simeq \colim_X K(\Proj(E(X))) \simeq \colim_X K(\Perf(E(X)))
\simeq K(\cat C)$. Instead of utilizing this approach, we present an
internal proof to the result in Barwick's algebraic $K$-theory machine
that we hope will be of independent interest.

\subsection{Outline of paper}

In section \ref{background}, we provide a brief introduction Barwick's
Waldhausen $\oo$-categories and highlight the pair structures of
interest for this paper.

In section \ref{basicprops}, we define weight structures and recount
some of their properties. We develop a yoga for manipulating weights
and recount results of Bondarko on generating weight structures on
categories. Finally, we compare the language of weight structures to
Waldhausen's formulation of the sphere theorem.

In section \ref{cellfiltrations}, we build the technical tools we
require for the proof of the main theorem. We introduce cellular
filtrations arising from weight structures and prove that they form
Waldhausen $\oo$-categories. We show that localizing these categories
at equivalences reflected from $\cat C$ induces a suitable model for
the $K$-theory of $\cat C$.

Section \ref{spheretheorem} comprises the proof of the main
theorem. We rely on tools developed in sections \ref{basicprops} and
\ref{cellfiltrations} as well as results of Barwick
\cite{barwick}. Finally, we prove our version of the ``$+ = Q$''
theorem: the $K$-theory of the heart of a weight structure admits a
description analogous to Quillen's plus construction.

Finally, in section \ref{examples}, we enumerate examples of weight
structures and explain applications of the main theorem in each
case. The Gillet--Waldhausen Theorem and the Resolution Theorem appear
as consequences of the main theorem. Several conjectural examples are
mentioned that merit exploration in future work.

\subsection*{Acknowledgments}
  I thank Ben Antieau, Clark Barwick, David Ben-Zvi, and David Gepner
  for helpful conversations about this paper. I especially thank
  Andrew Blumberg for leading me to this problem and providing
  guidance throughout its development.

\section{Towards Waldhausen $\oo$-categories}\label{background}
Algebraic $K$-theory is a functor which takes homotopical data as
input and produces a spectrum.

\begin{definition}[\cite{1126}*{1.2}]
  A \term{Waldhausen category} is a pointed category $\cat C$ equipped
  with subcategories, $w\cat C$ (the \term{weak equivalences}, denoted
  $\wto$) and $cof \cat C$ (the \term{cofibrations}, denoted $\cof$),
  satisfying the following axioms
  \begin{itemize}
  \item[(C1)] The isomorphisms of $\cat C$ are contained in $cof \cat C$.
  \item[(C2)] For every object $A\in \cat C$, the unique map $* \to A$
    is a cofibration.
  \item[(C3)]\label{axiomC3} For every diagram
    \begin{equation*}
      \begin{tikzcd}
        A & B \ar[l, tail] \ar[r] & C
      \end{tikzcd}
    \end{equation*}
    the pushout $A\cup_B C$ exists in $\cat C$ and $C\to A\cup_B C$ is
    a cofibration.
  \item[(W1)] The isomorphisms of $\cat C$ are contained in $w\cat C$.
  \item[(W2)]\label{axiomW2} For each commutative diagram
    \begin{equation*}
      \begin{tikzcd}
        A \ar[d, "\sim"] & B \ar[l, tail] \ar[r] \ar[d, "\sim"]& C
        \ar[d, "\sim"] \\
        A' & B' \ar[l, tail] \ar[r] & C'
      \end{tikzcd}
    \end{equation*}
    the induced map on pushouts $A\cup_B C\to A'\cup_{B'} C'$ is a
    weak equivalence.
  \end{itemize}
\end{definition}

Waldhausen categories encode homotopical information. Specifically,
they encode the data of which quotients (or pushouts) in the category
are homotopically meaningful. Axiom (C3) guarantees that
pushouts of cofibrations exist while axiom (W2) implies that
the weak equivalence type of those pushouts is not altered when the
diagram is replaced by a weakly equivalent one.

The algebraic $K$-theory of a Waldhausen category $\cat C$ is
constructed out of a simplicial object encoding information about a
sequences of cofibrations in $\cat C$ and their successive
quotients. This is Waldhausen's $S_\bullet$-construction, and the
algebraic $K$-theory of $\cat C$ is defined to be the space
\begin{equation*}
  K(\cat C) \coloneq \loops |w S_\bullet \cat C|
\end{equation*}
which deloops to form the connective $K$-theory spectrum which we will
also denote $K(\cat C)$ \cite{1126}*{1.3}.

In essence, the algebraic $K$-theory functor takes as input some
category with a notion of homotopy theory and information about the
homotopically-meaningful quotients in that category. The modern
perspective is that any model for a category equipped with a version
of homotopy theory (such as a model category or a Waldhausen category)
has an underlying $\oo$-category. In this paper, we make use of
Barwick's construction of algebraic $K$-theory \cite{barwick} for
$\oo$-categories.

Specifically, Barwick defines $\oo$-categorical analogs for Waldhausen
categories which are the input to his $K$-theory machine. These
Waldhausen $\oo$-categories are \term{pairs} of $\oo$-categories (\cf
\cite{barwick}*{Definition 1.11}) $(\cat C, \cat C_\dagger)$ subject
to several axioms. $\cat C_\dagger$ is a subcategory of $\cat C$ that
we will refer to as the \term{ingressive} morphisms of $\cat C$. The
ingressive morphisms are required to contain the maximal Kan
subcomplex $\iota \cat C$, and the \term{Waldhausen} structure on this
pair requires a zero object to exist in $\cat C$, all maps from zero
objects to be ingressive, and pushouts of ingressive morphisms to be
ingressive. Functors which respect Waldhausen $\oo$-categories are
called \term{exact} and are required to take zero objects to zero
objects, ingressives to ingressives, and pushouts squares of
ingressives to pushout squares. (For a precise definition, \cf
\cite{barwick}*{Definition 2.7}.) 

An $\oo$-category with the minimal pair structure
$\cat C_\dagger = \iota \cat C$ will be Waldhausen when it is a
contractible Kan complex. An $\oo$-category with the maximal pair
structure $\cat C_\dagger = \cat C$ will be Waldhausen when it has a
zero object and admits finite colimits. A stable $\oo$-category is
always a Waldhausen $\oo$-category via the maximal pair structure. The
relative nerve can produce a Waldhausen $\oo$-category from an
ordinary Waldhausen category by taking $N(\cat C, w\cat C)_\dagger$ to
be the smallest subcategory of $N(\cat C, w\cat C)$ that contains the
equivalences and the edges of $N\cat C$ corresponding to cofibrations
in $cof \cat C$. Barwick's $K$-theory of this associated Waldhausen
$\oo$-category agrees with Waldhausen's $K$-theory for the original
Waldhausen category.

In this paper, we are chiefly interested in the $K$-theory of a stable
$\oo$-category $\cat C$ and a specific subcategory $\cat C_{\heart w}$
which is not closed under suspension. $\cat C$ will have the maximal
pair structure where all maps are ingressives. $\cat C_{\heart w}$
will have the pair structure where maps are ingressives if they have
cofiber in $\cat C_{\heart w}$.

\section{Weight structures on stable $\oo$-categories}\label{basicprops}
In this section, we define weight structures and provide an overview
of their basic properties.

\subsection{Definitions}

\begin{definition}[{\cite{Bondarko}}]
  Let $\cat T$ be a triangulated category. A \emph{weight structure}
  $w$ on $\cat T$ is a pair of full subcategories $\cat T_{w\leq 0}$
  and $\cat T_{w \geq 0}$ (closed under retract and finite coproducts)
  satisfying the following properties. We adopt the notation that
  $\cat T_{w\leq n} \coloneq \Sigma^n \cat T_{w\leq 0}$ and $\cat
  T_{w\geq n} \coloneq \Sigma^n \cat T_{w\geq 0}$.
  \begin{enumerate}
  \item We have the inclusions $\cat T_{w\geq 1} \subseteq \cat
    T_{w\geq 0}$ and $\cat T_{w\leq -1} \subseteq \cat T_{w\leq 0}.$
  \item (Orthogonality) For $X \in \cat T_{w\leq 0}$ we have
    $\Hom_{\cat T}(X, Y) = 0$ for any $Y\in \cat T_{w \geq 1}$.
  \item (Weight decomposition) For any $X\in \cat T$ there exists a
    distinguished triangle
    \begin{equation*}
      \begin{tikzcd}
        X' \ar[r] & X \ar[r] & X''
      \end{tikzcd}
    \end{equation*} 
    with $X' \in \cat T_{w\leq 0}$ and $X''\in \cat T_{w\geq 1}$.
  \end{enumerate}
  
  We will call $\cat T_{\heart w} \coloneq \cat T_{w\leq 0}\cap \cat
  T_{w\geq 0}$ the \emph{heart} of the weight structure.
\end{definition}
We have chosen to use the \emph{homological} sign convention for weight
structures and all examples and statements will use that
convention. Some papers in the literature (\eg
\cite{Bondarko}) opt to use cohomological signs, writing
$\cat T^{w\geq 0}$ for what we denote by $\cat T_{w\leq 0}$. Due to a
lack of consensus in the literature, we use the convention that
appears more agreeable for homotopy theorists.

As defined, weight structures are overdetermined. Each subcategory of
the weight structure $\cat T_{w\leq 0}$ or $\cat T_{w\geq 0}$
determines the other by orthogonality (\cf proposition
\ref{mappingprop}). That is, $\cat T_{w\leq 0}$ is precisely the full
subcategory on objects $X$ with $\cat T(X,Y)=0$ for all
$Y\in \cat T_{w\geq 1}$. By translating (through suspension), we can
provide weight decompositions of an object $X$ at any degree. That is,
the degree-zero decomposition for $\Sigma^{-n} X$,
$A \to \Sigma^{-n} X \to B$ provides a degree-$n$-decomposition
$\Sigma^n A \to X \to \Sigma^n B$ with
$\Sigma^n A \in \cat T_{w\leq n}$ and
$\Sigma^n B \in \cat T_{w\geq n+1}$.

Note that by suspending and desuspending we can provide weight
decompositions of an object $X$ at any degree. That is, the
degree-zero decomposition for $\Sigma^{-n} X$,
$A \to \Sigma^{-n} X \to B$ provides a degree-$n$-decomposition
$\Sigma^n A \to X \to \Sigma^n B$ where
$\Sigma^n A \in \cat D_{w\leq n}$ and
$\Sigma^n B \in \cat C_{w\geq n+1}$.

The homotopy category of a stable $\oo$-category is triangulated
\cite{HA}*{1.1.2.15}.  Weight structures are defined on stable
$\oo$-categories by way of the homotopy category.
\begin{definition}
  Let $\cat C$ be a stable $\oo$-category. A \emph{weight structure}
  on $\cat C$ will be a weight structure on the triangulated category
  $h\cat C$.
\end{definition}

Full subcategories of $\oo$-categories are specified by full
subcategories of their homotopy categories. A weight structure on a
stable $\oo$-category $\cat C$ is equivalently defined by two full
$\oo$-subcategories $\cat C_{w \leq 0}$ and $\cat C_{w \geq 0}$ (with
$\cat C_{w\geq n}$ and $\cat C_{w\leq n}$ defined as above) and
(co)fiber sequences for each object
\begin{equation*}
  \begin{tikzcd}
    A \ar[r] \ar[d] & X \ar[d] \ar[r] & * \ar[d] \\ * \ar[r] & B
    \ar[r] & \Sigma A
  \end{tikzcd}
\end{equation*}
where both squares are pushouts in $\cat C$ and
$A\in \cat C_{w\leq n}$ and $B\in \cat C_{w \geq n+1}$. Furthermore,
the orthogonality condition requires that $\pi_0 \cat C(A,B) = 0$ for
all $A\in \cat C_{w \leq n}$ and $B\in\cat C_{w\geq n+1}$. We will
further abuse notation by referring to
$\cat C_{\heart w} \coloneq \cat C_{w\leq 0}\cap \cat C_{w\geq 0}$ as
the heart of the weight structure on $\cat C$.
\begin{remark}
  Throughout this paper, $\cat C_{\heart w}$, the heart of a weight
  structure on a stable $\oo$-category, is defined to be a full
  $\oo$-subcategory of $\cat C$. This conflicts with the notation used
  in triangulated categories, where the heart of the weight structure
  is the subcategory of the homotopy category of $\cat C$. The latter
  we consider $\pi_0 \cat C_{\heart w}$.
\end{remark}

\begin{definition}
  We say that a weight structure $w$ on a triangulated category
  $\cat T$ is \emph{non-degenerate} if
  $\bigcap_{n \to \oo} \cat T_{w\geq n} = 0$ and
  $\bigcap_{n\to -\oo} \cat T_{w\leq n} = 0$.
\end{definition}
A weight structure $w$ on a stable $\oo$-category $\cat C$ is
non-degenerate precisely when $\bigcap_{n \to \oo} \cat C_{w\geq n}$
and $\bigcap_{n\to -\oo} \cat C_{w\leq n}$ are equivalent to the
subcategory of zero objects $\cat C_{0}$ in $\cat C$,

Throughout the paper, we will view the stable $\oo$-category $\cat C$
as a Waldhausen $\oo$-category equipped with the maximal pair
structure in which all edges are ingressive. The heart
$\cat C_{\heart w}$ is given a sub-Waldhausen structure: an edge is
ingressive only when its cofiber also lives in $\cat C_{\heart w}$.

Several examples of weight structures are worked out in section
\ref{examples}. We mention the following example for the reader to
keep in their mind.
\begin{example}
  The \emph{Postnikov weight structure} on finite spectra takes
  \begin{equation*}
    \Sp_{w\geq
      n} = \{E : \pi_*(E) = 0, \forall *<n\}
  \end{equation*}
  that is, the $n$-connective spectra, and
  \begin{equation*}
    \Sp_{w\leq n} = \{E :
    \HZ_*(E) = 0, \forall * > n \textrm{~and~} \HZ_n(E)
    \textrm{~free}\}.
  \end{equation*}
  The heart $\Sp_{\heart w}$ consists spectra weakly equivalent to
  finite wedge sums of copies of $S^0$.
\end{example}

Weight structures generalize cellular structures. $\cat C_{w\leq n}$
is the subcategory of ``cell-less-than-$n$'' objects and $\cat
C_{w\geq n}$ is the ``cell-greater-than-$n$'' subcategory. The weight
decomposition is analogous to the inclusion of an $n$-skeleton into
$X$.

Weight structures are not formally dual to $t$-structures. The
decompositions arising from $t$-structures are unique as can be
summarized by the existence of a localizing ``truncation'' functor
$\cat C \to \cat C_{t\geq n}$ for all $n$. In contrast, weight
structure decompositions have no such unicity: there can be many
choices of $n$-skeleta for each object. Instead for weight
decompositions, we have the following result in the homotopy category
$h\cat C$.
\begin{proposition}
  Let $A_{w\leq n} \to X \to B_{w\geq n+1}$ denote a weight
  decomposition of $X$ at degree $n$. If $n\leq m$ then we get maps
  \begin{equation*}
    \begin{tikzcd}
      A_{w\leq n} \ar[r] \ar[d, "\exists", dotted] & X \ar[r] \ar[d,
      equals] & B_{w\geq n+1} \ar[d, "\exists", dotted] \\
      A'_{w\leq m} \ar[r] & X \ar[r] & B'_{w\leq m}
    \end{tikzcd}
  \end{equation*}
  making the diagram commute. If $n < m$ then the induced maps are
  unique.
\end{proposition}
As a consequence of this proposition, maps between skeleta are
determined up to homotopy provided they map into a strictly ``higher''
skeleton.

\begin{remark}
  As a subcategory of $C$, $\cat C_{w \leq n}$ is closed under forming
  fibers. Likewise, $\cat C_{w\geq n}$ is closed under forming
  cofibers. If we consider $\cat C$ as a maximal exact $\oo$-category
  in the sense of \cite{barwickheart}, $\cat C_{w\geq n}$ is a
  Waldhausen $\oo$-subcategory and $\cat C_{w\leq n}$ is a
  coWaldhausen $\oo$-subcategory.
\end{remark}

\begin{definition}
  The weight structure on $\cat C$ is \emph{bounded} if
  \begin{equation*}
    \bigcup_{n\geq 0} \left( \cat C_{w\geq -n} \cap \cat C_{w\leq n}
    \right) = \cat C.
  \end{equation*}
\end{definition}

The weight structure defined above on the category of finite spectra
is bounded. The same weight structure on the category of (not
necessarily finite) spectra is not bounded but its ``bounded closure''
$\bigcup_{n\geq 0} \left( \cat C_{w\geq -n} \cap \cat C_{w\leq n}
\right)$ consists of the finite spectra.

\subsection{Properties of weight structures}
In this section we establish some basic properties of weight
structures and how weights interact with forming fibers and
cofibers. The punchline of the section is that a weight structure
provides cellular decompositions of objects.

When convenient, we will will use subscripts to denote the weights of given
objects. That is, $A_{w\leq n}$ will denote that the object
$A$ has weight $w\leq n$ in $\cat C$.

\begin{proposition}\label{mappingprop} %
  $\cat C_{w\leq n}$ determines $\cat C_{w\geq n+1}$ and vice-versa:
  $\cat C_{w\geq n+1}$ is precisely those $X \in \cat C$ with
  $\Hom(Y,X)=0$ for all $Y\in \cat C_{w\leq n}$.
\end{proposition}
\begin{proof}
  If $X$ is as given in the statement, then it's weight decomposition
  at $n$ is a fiber sequence
  \begin{equation*}
    \begin{tikzcd}
      X' \ar[r] & X \ar[r] & X''
    \end{tikzcd}
  \end{equation*}
  but $X' \in \cat C_{w \leq n}$ so $X' \to X$ is the zero map in
  $h\cat C$. Thus $X \simeq X''$. The proof is identical to show that
  an object lives in $\cat C_{w\geq n+1}$.
\end{proof}

\begin{proposition}\label{PropGeqCofibers}
  $\cat C_{w\geq n}$ is closed under retracts and cofibers.
\end{proposition}
\begin{proof}
  Say $X$ is a retract of $Y$ in $\cat C$ and $Y$ lies in $\cat
  C_{w\leq n}$. Then it will be as well in $h\cat C$. In particular,
  fix $i:X\to Y$ and $r:Y\to X$ with $r\comp i = \id_X$ in $h\cat
  C$. For any $Z \in \cat C_{w\geq n+1}$ we have induced maps
  \begin{equation*}
    \begin{tikzcd}
      h\cat C(X,Z) \ar[r, "r^*"] & h\cat C(Y,Z) \ar[r, "i^*"] & h\cat
      C(X,Z)
    \end{tikzcd}
  \end{equation*}
  whose composite must be the identity. This demonstrates $h\cat C(X,Z)$
  as a retract of $h\cat C(Y,Z)$. The latter is trivial so the former
  must be as well. The previous proposition concludes that $\cat
  C_{w\geq n}$ is closed under retracts.

  Now suppose $X$ and $Y$ both live in $\cat C_{w\geq n}$ and
  \begin{equation*}
    \begin{tikzcd}
      X \ar[r, "f"] & Y \ar[r] & \cofiber(f)
    \end{tikzcd}
  \end{equation*}
  is a cofiber sequence in $\cat C$. Rotating forward, we have a
  cofiber sequence $Y \to \cofiber(f) \to \Sigma X$. Let $Z$ be any
  object of $\cat C_{w\geq n+1}$. Since $h\cat C(-,Z)$ carries cofiber
  sequences to fiber sequences, we have the following fiber sequence.
  \begin{equation*}
    \begin{tikzcd}
      h\cat C(\Sigma X, Z) \ar[r] & h\cat C(\cofiber(f),Z) \ar[r] &
      h\cat C(Y, Z)
    \end{tikzcd}
  \end{equation*}
  The axioms for a weight structure tell us that
  $\Sigma X \in \cat C_{w\geq n+1} \subseteq \cat C_{w\geq n}$ and
  thus all terms of this sequence must be trivial as the outer two
  are. The previous proposition concludes the proof.
\end{proof}

\begin{proposition}
  $\cat C_{w\leq n}$ is closed under retracts and forming fibers.
\end{proposition}
\begin{proof}
  Both proofs are nearly identical to those for the previous
  proposition when one replaces $h\cat C(-,Z)$ with $h\cat
  C(Z,-)$. $h\cat C(Z,-)$ carries fiber sequences to fiber sequences,
  and since $\cat C$ is stable, fiber and cofiber sequences
  coincide. Backing up a fiber sequence $\fiber(f) \to X \to Y$ to
  produce $\Sigma^{-1} Y \to \fiber(f) \to X$ and similar arguments
  conclude the proof.
\end{proof}

\begin{remark}
  $\cat C_{w\geq n} \subset \cat C$ are Waldhausen subcategories of
  $\cat C$. If $\cat C$ is an exact stable $\oo$-category in the sense
  of \cite{barwickheart} with a weight structure then $\cat
  C_{w\leq n} \subset \cat C$ are coWaldhausen subcategories.
\end{remark}

The following is a lemma about triangulated categories that is
surprisingly useful for manipulating weight structures.
\begin{lemma}[{\cite{MR2746283}*{1.4.1}}]\label{lem1}
  Let $X\to A \to B \to \Sigma X$ and $X'\to A'\to B'\to \Sigma X$ be
  two distinguished triangles in a triangulated category $h\cat C$.
  \begin{enumerate}
  \item If $h\cat C(B, \Sigma A') = 0$ then for any $g:X\to X'$ there
    exist $h:A\to A'$ and $i:B\to B'$ completing $g$ to a map of
    distinguished triangles.
  \item If furthermore $h\cat C(B, A')=0$ then $h$ and $i$ are unique.
  \end{enumerate}
\end{lemma}
\begin{proof}
  By the axioms for a triangulated category it suffices to provide one
  of the two desired maps. Applying $h\cat C(B,-)$ to the second
  distinguished triangle yields the following exact sequence.
  \begin{equation*}
    \begin{tikzcd}
      h\cat C(B,A') \ar[r] & h\cat C(B,B') \ar[r] & h\cat C(B, \Sigma
      X') \ar[r] & h\cat C(B, \Sigma A')
    \end{tikzcd}
  \end{equation*}
  The assumption $h\cat C(B,\Sigma A')=0$ lets us lift the composite
  $B\to \Sigma X \xrightarrow{\Sigma g}{} \Sigma X'$ to $i :B\to
  B'$. If the second assumption holds this map is determined uniquely
  by $g$.
\end{proof}

Now let $A_{w\leq n} \to X \to B_{w\geq n+1}$ and
$A'_{w\leq m} \to X \to B_{w\geq m+1}$ denote two weight decomposition
of $X$ at degrees $n$ and $m$, respectively.
\begin{corollary}\label{cor1}
  There are maps $a:A_{w \leq n} \to A'_{w \leq m}$ and $b:B_{w \geq
    n+1} \to B'_{w\geq m+1}$ that assemble into a map of distinguished
  triangles
  \begin{equation*}
    \begin{tikzcd}
      A_{w\leq n} \ar[r] \ar[d, "a"] & X \ar[d, equals] \ar[r] & B'_{w\geq n+1} \ar[d, "b"] \\
      A'_{w\leq m} \ar[r] & X \ar[r] & B'_{w\geq m+1}
    \end{tikzcd}
  \end{equation*}
  whenever $m \geq n$. If $m \geq n+1$ then these maps are unique in
  $h\cat C$.
\end{corollary}
\begin{proof}
  Apply Lemma \ref{lem1} to the sequences provided with the map
  $\id:X\to X$.
\end{proof}
Note that the maps are unique up to choice of the two decompositions
(which are not \emph{a priori} unique).

\begin{corollary}
  If $X$ has weight $w\geq n$ then for all $k \geq n-1$ any weight
  decomposition
  \begin{equation*}
    \begin{tikzcd}
      A_{w\leq k} \ar[r] & X_{w\geq n} \ar[r] & B_{w\geq k+1}
    \end{tikzcd}
  \end{equation*}
  is equivalent to the trivial decomposition
  \begin{equation*}
    \begin{tikzcd}
      * \ar[r] & X_{w\geq n} \ar[r] & X_{w\geq n}
    \end{tikzcd}
  \end{equation*}
\end{corollary}
\begin{proof}
  Since $k+1 \geq n$, $X$ lies in $\cat C_{w\geq k+1}$. By the lemma
  we have maps between the two sequences. The unicity of maps into and
  out of $*$ makes these unique. Thus $A \simeq *$ and $B \simeq X$.
\end{proof}

\begin{corollary}
  If $X$ has weight $w\leq n$ then for all $k \geq n$ any weight
  decomposition
  \begin{equation*}
    \begin{tikzcd}
      A_{w\leq k} \ar[r] & X_{w\leq n} \ar[r] & B_{w\geq k+1}
    \end{tikzcd}
  \end{equation*}
  is equivalent to the trivial decomposition
  \begin{equation*}
    \begin{tikzcd}
      X_{w\leq n} \ar[r] & X_{w\leq n} \ar[r] & *
    \end{tikzcd}
  \end{equation*}
\end{corollary}
\begin{proof}
  The proof is identical to the last corollary.
\end{proof}

\begin{proposition}\label{prop2}
  If $X$ has weight $w\geq n$ then for all $k \geq n$ any weight
  decomposition
  \begin{equation*}
    \begin{tikzcd}
      A_{w\leq k} \ar[r] & X_{w\geq n} \ar[r] & B_{w\geq k+1}
    \end{tikzcd}
  \end{equation*}
  has $A$ in $\cat C_{w\geq n}$.
\end{proposition}
\begin{proof}
  Note that $\Sigma^{-1} B$ has weight $w\geq k$ by the axioms. By
  assumption, this means that $\Sigma^{-1}B \in \cat C_{w \geq
    n}$. Thus, rotating back the fiber sequence for the decomposition
  yields the fiber sequence
  \begin{equation*}
    \begin{tikzcd}
      (\Sigma^{-1} B)_{w\geq n} \ar[r] & A \ar[r] & X_{w\geq n}
    \end{tikzcd}
  \end{equation*}
  which demonstrates $A\in \cat C_{w\geq n}$ by Proposition
  \ref{mappingprop}.
\end{proof}
\begin{proposition}\label{PropCofibBound}
  If $X$ has weight $w\leq n$ then for all $k \leq n-1 $ any weight
  decomposition
  \begin{equation*}
    \begin{tikzcd}
      A_{w\leq k} \ar[r] & X_{w\leq n} \ar[r] & B_{w\geq k+1}
    \end{tikzcd}
  \end{equation*}
  has $B$ in $\cat C_{w\leq n}$.
\end{proposition}
\begin{proof}
  The proof is identical to that for the previous proposition.
\end{proof}

\begin{proposition}\label{prop_maps_detected_levelwise}
  Suppose $\cat C$ is a stable $\oo$-category with a non-degenerate
  weight structure. Maps are detected by maps into or out of the heart
  in the following sense.

  For any $X\in \cat C$, $\pi_0h\cat C(X,Y)=0$ for all
  $Y\in \cat C_{w\geq 0}$ if and only if $\pi_0 h\cat C(X,Q)=0$ for
  all $Q \in \cat C_{w=i}$ for $i \geq 0$.

  Likewise, for any $Y\in \cat C$, $\pi_0 h\cat C(X,Y) = 0$ for all
  $X \in \cat C_{w\leq 0}$ if and only if $\pi_0 h\cat C(Q, Y) = 0$
  for all $Q \in \cat C_{w=i}$ with $i\leq 0$.
\end{proposition} %
\begin{proof}
  Both proofs are identical so we check the first.  The forward
  direction is trivial. For the reverse implication, fix a map
  $f:X\to Y$ in $h\cat C$. Pick a weight decomposition at degree $0$
  for $Y$. By proposition \ref{prop2}, this takes the form
  \begin{equation*}
    \begin{tikzcd}
      A_{w = 0} \ar[r] & Y_{w\geq 0} \ar[r] & B_{w \geq 1}
    \end{tikzcd}
  \end{equation*}
  and applying $\pi_0 h\cat C(X,-)$ produces a long exact sequence on
  mapping groups. By assumption,
  $\pi_0 h\cat C(X,A) = \pi_0h\cat C(X,\Sigma A) = 0$, so we have that
  $\pi_0 h\cat C(X,Y)$ is isomorphic to $\pi_0 h\cat C(X,B)$ and $B$
  has weight %
  $w\geq 1$. We can iterate this argument: replace $Y$ with $B$ in
  this argument and take a weight decomposition at degree
  $1$. Inductively, we can conclude that $\pi_0 h\cat C(X,Y)$ is
  isomorphic to $\pi_o h\cat C(X,\twiddle B)$ where $\twiddle B$ can
  be constructed in an arbitrarily high weight $w\geq n$. As
  $n\to \oo$, we conclude that $\pi_0 h\cat C(X,Y)\isom 0$ as only
  zero objects have arbitrarily high weights due to the non-degeneracy
  of $w$.
\end{proof}

\subsection{Generating weight structures}\label{generatingweights}

Suppose $\cat C$ is a stable $\oo$\hyp{}category and $\cat H$ is a
subcategory in $\cat C$. A natural question is whether there exists a
weight structure on $\cat C$ with $\cat H$ as its heart. We will
require that $\cat H$ is closed under retracts and finite coproducts.

\begin{definition}
  We say that $\cat H$ \emph{weakly generates} $\cat C$ if
  $X\in \cat C$ and $$\pi_0 h\cat C(\Sigma^n S, X) = 0$$ for all
  $n\in \Z$ and for all $S\in \cat H$, then $X$ is a zero object in
  $\cat C$.
\end{definition}

\begin{definition}
  We say that $\cat H$ is \emph{negative} if for all $n>0$ we have
  $$\pi_0 h\cat C(S, \Sigma^n S') = 0$$ for all $S, S' \in \cat H$.
\end{definition}
For spectrally-enriched categories, Blumberg and Mandell introduce a
very similar notion to a negative subcategory, namely a
\emph{connective class}. This definition is essential to their form of
the sphere theorem which is discussed in
\cite{2011arXiv1111.4003B}*{3.4}.

\begin{proposition}[{\cite{MR2746283}*{4.3.2.III(ii) and 4.5.2}}]\label{generatingweightstheorem}
  Suppose the objects of $\cat H$ are compact, $\cat H$ is negative,
  and $\cat H$ weakly generates $\cat C$. Suppose further that all
  finite cell complexes constructed from $\cat H$ exist in $\cat C$.

  Let $\cat C_-$ be the full subcategory of $\cat C$ of objects $X$ so
  that $\forall S\in \cat H$ there exists a $N\in \Z$ so that
  $\pi_0 \cat C(Y, \Sigma^n S) = 0$ for all $n > N$. Then $\cat C_-$
  admits a weight structure with $\cat H$ contained in its heart.
\end{proposition}%

We introduce two examples of such weight structures now and provide a
more detailed discussion in section \ref{examples}.

\begin{example}
  Let $R$ be a commutative ring. Let $\cat C= \Ch_R$ denote the stable
  $\oo$-category of bounded-above chain complexes of
  finitely\hyp{}generated $R$\hyp{}modules. There is a weight structure on
  $\Ch_R$ where $\Ch_{R, w\geq 0}$ contains complexes whose homology
  is concentrated in non-negative degrees, and $\Ch_{R, w\leq 0}$
  contains complexes which are quasi-isomorphic to complexes of
  projectives whose homology is non-positive degrees (or,
  equivalently, complexes concentrated in non-positive degrees and
  projective in degree $0$). The heart of this weight structure is the
  finitely-generated projective $R$-modules included as complexes
  concentrated in degree $0$.
\end{example}

\begin{example}
  Let $\Sp^\fin$ denote the stable $\oo$-category of finite
  spectra. $\Sp^\fin$ admits a weight structure generated by the
  sphere spectrum $S^0$. $\Sp^\fin_{w\geq 0}$ consists of all
  connective spectra and $\Sp^\fin_{w\leq 0}$ consists of spectra
  whose integral homology is concentrated in non-positive
  degrees. These are precisely spectra which occur as $k$-skeleta for
  other spectra for $k\leq 0$. %
\end{example}

\begin{definition}
  Let $\cat C$ be a stable $\oo$-category equipped with a weight
  structure and a $t$-structure. We say that these structures are
  \emph{left adjacent} (respectively, \emph{right adjacent}) if
  $\cat C_{t\geq 0} = \cat C_{w\geq 0}$ (respectively,
  $\cat C_{t\leq 0} = \cat C_{w\leq 0}$).
\end{definition}

When a weight structure is left adjacent to a $t$-structure, the
orthogonality relations of the two structures interact to permit more
specific descriptions of the hearts $\cat C_{\heart w}$ and
$\cat C_{\heart t}$.

If $X$ is in the heart of the weight structure, then $X$ is in
$\cat C_{w\geq 0} = \cat C_{t\geq 0}$ so for any $Y$ in
$\cat C_{t\leq -1}$ we have $\pi_0 \cat C(X,Y)=0$. Likewise, $X$ also
lies in $\cat C_{w \leq 0}$ so the orthogonality relations for the
weight structure make it admit no maps (on $\pi_0$) to
$\cat C_{w \geq 1} = \cat C_{t\geq 1}$. As a consequence of the
overdetermined nature of weight and $t$-structures (\cf Proposition
\ref{mappingprop}), the heart of the weight structure
consists of precisely those objects $X$ with $\pi_0\cat C(X,Y)=0$ for
$Y \in \cat C_{t\geq 1} \cup \cat C_{t\leq -1}$.

A similar analysis can be applied to $\cat C_{\heart t}$ to deduce
that the heart can be detected by $\cat C_{w\leq -1}$ and
$\cat C_{w\geq 1}$. Together with Proposition
\ref{mappingprop}, we arrive at the following
description.
\begin{proposition}
  Suppose $\cat C$ admits left adjacent weight- and $t$\hyp{}structures. $X$
  is in $\cat C_{\heart w}$ if and only if
  $\pi_0\cat C(X,\Sigma^i B) = 0$ for all $B\in \cat C_{\heart t}$ for
  $i\neq 0$. Likewise, $Y$ is in $\cat C_{\heart t}$ if and only if
  $\pi_0\cat C(\Sigma^i A,Y)=0$ for all $A\in \cat C_{\heart w}$ for
  $i\neq 0$.
\end{proposition}

\begin{example}
  In the category of finite spectra, the cellular weight structure is
  left adjacent to the Postnikov $t$-structure. The heart of the
  weight structure consists of wedge sums of the sphere spectrum and
  the heart of the $t$-structure is the Eilenberg--Mac Lane
  spectra. The proposition notes that the former are (equivalently)
  the spectra with whose cohomology is concentrated in degree $0$ (for
  all $HG$), while the latter are precisely those spectra with
  homotopy groups concentrated in degree $0$.
\end{example}

\subsection{On weights and Waldhausen's sphere theorem}

This section places Waldhausen's original sphere theorem within our
setting of weight structures on stable $\oo$-categories. Proposition
\ref{implieswald} proves that our theorem from section
\ref{spheretheorem} generalizes Waldhausen's. We take the rest of the
section to explore the limits how analogous language can be lifted
from Waldhausen's setting to the world of weight structures.

As originally formulated in \cite{1126}*{1.7}, Waldhausen's sphere
theorem applies to a Waldhausen category $\cat C$ equipped with a
cylinder functor that satisfies the cylinder axiom. The category must
be further equipped with a $\Z$-graded homology functor $H_*$ which
carries cofiber sequences to long exact sequences in some abelian
target category. Furthermore, weak equivalences in $\cat C$ are
required to be precisely isomorphisms on homology. Finally, the
hypothesis for the sphere theorem is that any $m$-connected map
$X\to Y$ (with respect to $H_*$) can be factored as
\begin{equation*}
  \begin{tikzcd}
    X_m \ar[r, tail] & X_{m+1} \ar[r, tail] & \cdots \ar[r, tail] &
    X_n \ar[r, "\simeq"] & Y
  \end{tikzcd}
\end{equation*}
where the quotients $X_{k+1}/X_{k}$ are all homology spheres of
dimension $k+1$. In this case, the sphere theorem says that the
$K$-theory of the stabilization of $\cat C$ (under the suspension
defined by the cylinder functor) is equivalent to the $K$-theory of
the stabilization (under suspension again) of the homology spheres. In
Waldhausen's context, a \emph{homology $n$-sphere} is an object $X$
whose homology $H_i(X)$ is $0$ unless $i=n$ and then lies in some
fixed full subcategory $\cat E$ of the abelian target category of
$H_*$ which is closed under extensions and retracts.

\begin{proposition}\label{implieswald}
  If $\cat C$ is a Waldhausen category satisfying the hypotheses of
  Waldhausen's sphere theorem, then the stable $\oo$-category
  $\Stab(\cat C)$ admits a bounded and non-degenerate weight structure
  whose heart is equivalent to the stabilized homology spheres in
  $\cat C$ if it has a set of compact generators which:
  \begin{itemize}
  \item generate the $\oo$-category under finite colimits,
  \item are homology $0$-spheres, and
  \item form a negative class in $\cat C$.
  \end{itemize}
\end{proposition}

Although Waldhausen does not require these additional assumptions,
they are true in the cases he studies.

We prove this proposition by defining a weight structure on
$\Stab(\cat C)$ where objects are in weight $w\geq n$ if their
homology is concentrated in degrees $* \geq n$ and in weight $w\leq n$
if their homology is concentrated in degrees $* \leq n$ and
$H_n(X) \in \cat E$. Under the hypotheses listed, we can generate a
weight structure on $\Stab(\cat C)$ using proposition
\ref{generatingweightstheorem}. The heart is precisely the homology
$n$-spheres as claimed.

We can transplant Waldhausen's language to the setting of weight
structures on stable $\oo$-categories. Specifically, we can view
weight structures as providing a language for discussing connectivity
of maps without specifying compact generators whose (co)homology
theories measure connectivity.

\begin{definition}
  A map $f:X\to Y$ in $\cat C$ with cofiber $Cf$ will be called
  $n$-connected if $Cf$ lives in $\cat C_{w\geq n+1}$.
\end{definition}

\begin{proposition}\label{prop3}
  The composite of two $n$\hyp{}connected maps is $n$\hyp{}connected.
\end{proposition}
\begin{proof}
  Say $f:X\to Y$ and $g:Y\to Z$ are $n$-connected. Write $D$ for the
  given pushout in the following diagram.
  \begin{equation*}
    \begin{tikzcd}
      X \ar[r, "f"] \ar[d] \po & Y \ar[d] \ar[r, "g"] \po & Z \ar[d] \\
      * \ar[r] & Cf_{w\geq n+1} \ar[r] \ar[d] \po & D \ar[d] \\
      & * \ar[r] & Cg_{w\geq n+1}
    \end{tikzcd}
  \end{equation*}
  here $Cf$ and $Cg$ are the respective cofibers and $D$ is evidently
  the cofiber of the composite $g\comp f$. Since the top and outer
  squares on the right of the diagram are pushouts, so is the lower
  square. The lower square induces a distinguished triangle $Cf_{w\geq
    n+1} \to D \to Cg_{w\geq n+1}$ in $h\cat C$ and thus $D$ lies in
  $\cat C_{w\geq n+1}$.
\end{proof}

This proposition implies that a weight decomposition at degree
$k\leq n$ is guaranteed to yield a degree-$k$ decomposition for $Y$ as
well after composing with an $n$-connected map $f:X\to Y$.
 
\begin{proposition}
  If the weight structure on $\cat C$ is bounded, any $n$-connected map
  $f:X\to Y$ factors
  \begin{equation*}
    \begin{tikzcd}
      X = X_n \ar[r] & X_{n+1} \ar[r] & \cdots \ar[r] X_m \ar[r,
      "\simeq"] & Y
    \end{tikzcd}
  \end{equation*}
  with $X_{k} / X_{k-1}$ in $\cat C_{w=k}$ for $n+1 \leq k \leq m$.
\end{proposition}
\begin{proof}
  We will induct up until the cofiber must be concentrated in weight
  $w=m$ due to boundedness of the weight structure. The induction
  essentially proceeds by providing a cellular filtration for $Cf$
  (see section \ref{cellfiltrations}). Fix a diagram in $\cat C$ for
  the cofiber sequence for $Cf$.
  \begin{equation*}
    \begin{tikzcd}
      X \ar[r, "f"] \ar[d] \po & Y \ar[d, "g"] \ar[r] \po & 0' \ar[d] \\
      0 \ar[r] & Cf \ar[r, "h"] & \Sigma X
    \end{tikzcd}
  \end{equation*}
  We will write $X_n=X$ to start the induction as $Cf$ has weight
  $w \geq n+1$ by assumption.

  Fix $A_k\to Cf \to B_{k+1}$ weight decompositions at $k$ for all
  $k$. Proposition \ref{prop2} tells us that $A_k$ lives in weight
  $n+1 \leq w \leq k$.  In particular, $A_{n+1}$ lives in
  $\cat C_{w=n+1}$. Fix a lift of the map $a:A_{n+1}\to Cf$ to
  $\cat C$. Form $X_{n+1}$ as the cofiber of the composite
  \begin{equation*}
    \begin{tikzcd}
      \Sigma^{-1} A_{n+1} \ar[r, "\Sigma^{-1} a"] & \Sigma^{-1} Cf \ar[r, "\Sigma^{-1} h"]
      & X_n
    \end{tikzcd}
  \end{equation*}
  By construction, the cofiber of the map $X_n \to X_{n+1}$ will be
  equivalent to $A_{n+1}$ which is in $\cat C_{w=n+1}$ as desired. It
  remains to show that there is an $(n+1)$-connected map $X_{n+1}\to
  Y$ to complete the induction.

  The composite $\Sigma^{-1} A_{n+1}\to Y$ is homotopic to the zero map
  because it factors through two consecutive maps in a cofiber
  sequence.
  \begin{equation*}
    \begin{tikzcd}
      \Sigma^{-1} A_{n+1} \ar[r] & \Sigma^{-1} Cf \ar[r] \ar[rr, bend right=50,
      ""{name=U, above}, "0"'] & X_n \ar[r] \ar[Rightarrow, to=U] & Y
    \end{tikzcd}
  \end{equation*}
  
  Thus $Y$ admits a map from $X_n$. This leads us to consider the
  following diagram.
  \begin{equation*}
    \begin{tikzcd}
      X_n \ar[r] \ar[d] & X_{n+1} \ar[r]\ar[d] & Y \ar[d] \\
      0 \ar[r] & A_{n+1} \ar[r] \ar[d] & Cf \ar[d] \\
      & 0 \ar[r] & Cf_{n+1}
    \end{tikzcd}
  \end{equation*}
  We know that the outer upper square is a pushout along with the
  upper left. This implies that the upper right one is as well. We
  form the lower square as the cofiber of the map $A_{n+1}\to Cf$. The
  outer right square is thus also a pushout and identifies the lower
  right square as the cofiber of $f_{n+1}:X_{n+1}\to Y$. The lower
  right square now tells us that $Cf_{n+1}$ lives in
  $\cat C_{w\geq n+2}$ as desired. The relevant cofiber sequence with
  weights marked is indicated below.
  \begin{equation*}
    \begin{tikzcd}
      (A_{n+1})_{w=n+1}\ar[r] & (Cf)_{w\geq n+1} \ar[r] & Cf_{n+1}
      \ar[r] & (\Sigma A_{n+1})_{w=n+2}
    \end{tikzcd}
  \end{equation*}
\end{proof}

\section{Bounded cell complexes}\label{cellfiltrations}
In this section, we define cellular weight filtrations and develop
some of their properties. The proof of our main theorem relies on
careful manipulation of these cellular filtrations. Throughout, we
will assume $\cat C$ is a stable $\oo$-category, viewed as a
Waldhausen $\oo$-category equipped with the maximal pair structure,
and $w$ is a bounded weight structure on $\cat C$.

\subsection{Definitions}

In preparation for the proof of the main theorem, we study
an ancillary object: the $\oo$-category of bounded cell complexes in
$\cat C$.
\begin{definition}
  Suppose $\cat C$ is a stable Waldhausen $\oo$-category. A
  \emph{relative cell complex} in $\cat C$ is a functor
  $A:(N\Z)^\sharp \to \cat C$ of Waldhausen $\oo$-categories so that
  any quotient $A_i/A_{i-1}$ is in $\cat C_{w=i}$. $\lim_\Z$ and
  $\colim_\Z$ define functors from the category of relative cell
  complexes to $\cat C$. A \emph{cell complex} will be a relative cell
  complex which $\lim_\Z$ takes to a zero object of $\cat C$. Write
  $\cell\cat C \subset \Fun_{\Wald_\oo}((N\Z)^\sharp, \cat C)$ for the
  full $\oo$-subcategory of cell complexes in $\cat C$.  We will write
  $A_{\oo}$ for $\colim_\Z A$ and $A_{-\oo}$ for $\lim_\Z A$ and will
  say that $A$ \emph{is a filtration for} $A_\oo$.
\end{definition}
By definition, all the morphisms $A_n \to A_m$ in the diagram for a
cell complex $A$ are ingressions in $\cat C$. Furthermore, two cell
complexes $A_\bullet$ and $B_\bullet$ in $\cell\cat C$ are equivalent
if there is a map between them that restricts levelwise to
equivalences in $\cat C$, \ie levelwise these edges must lie in
$i\cat C$.

Let $i_{\leq n}:\Z_{\leq n} \to \Z$ be the inclusion of the poset of
integers $\leq n$. $i_{\leq n}$ induces a functor
$i^*_{\leq n} :\Fun_{\Wald_\oo}((N\Z)^\sharp, C) \to
\Fun_{\Wald_\oo}((N\Z_{\leq n})^\sharp, C)$
which admits a left adjoint. The adjoint is induced by the map
$p_{\leq n}:\Z \to \Z_{\leq n}$ which is the identity on $\Z_{\leq n}$
and collapses all larger integers to $n$. Write $\tr_{n}$ for the
composite $i^*_{\leq n}\comp p^*_{\leq n}$, the degree-$n$ truncation
of a (relative) cell complex.  Likewise
$i_{\geq n}:\Z_{\geq n} \to \Z$ induces a functor on relative cell
complexes which admits a right adjoint induced by the map
$p_{\geq n}:\Z\to \Z_{\geq n}$ which is the identity $\geq n$ and
collapses all integers below $n$ to $n$. Write $\cotr_n$ for the
composite $p_{\geq n}^*\comp i^*_{\geq n}$, the degree-$n$
cotruncation.

\begin{definition}
  We call a (relative) cell complex $A$ \emph{bounded} if
  $A \simeq \tr_n A \simeq \cotr_m A$ for some finite $n$ and
  $m$. Write $\cell^\bdd\cat C$ for the full $\oo$-subcategory of
  $\cell\cat C$ on bounded cell complexes. If a bounded complex $A$ is
  equivalent to its $n$-truncation $\tr_n A \simeq A$, then we say that
  $A$ has \emph{degree $\leq n$}. Write $\cell_n^\bdd\cat C$ for the
  full $\oo$-subcategory of $\cell^\bdd\cat C$ on degree $\leq n$
  cell complexes.
\end{definition}
A cell complex $A$ with $A\simeq \cotr_n A$ must have
\begin{equation*}
  \lim_\Z A \simeq \lim_Z \cotr_n A \simeq A_n
\end{equation*}
as a zero object. %
Thus, bounded cell complexes are finite-stage cellular constructions
in the weight structure on $\cat C$ that begin with a zero object.
The subcategories $\cell_n^\bdd \cat C$ filter $\cell^\bdd \cat C$.
That is, under the inclusion maps
$\colim_n \cell_n^\bdd \cat C \simeq \cell^\bdd \cat C$.

\begin{proposition}
  If $A$ is a bounded cell complex in $\cat C$, then $A_i$ is in
  $\cat C_{w\leq i}$ for all $i$. If $A_n$ is in $\cat C_{w \geq n}$
  then $A_{n-1}$ is in $\cat C_{w\geq n-1}$ as well. In particular, if
  $A\in \cell_n^\bdd \cat C$ and $A_\oo \in \cat C_{w=n}$ then
  $A_i\in \cat C_{w=i}$ for all $i$.
\end{proposition}
\begin{proof}
  Induct up the filtration starting with a zero object $A_{\leq -N}$
  in $\cat C_{\leq -N}$ as in \S\ref{basicprops}. For the second part
  of the proposition, use proposition \ref{PropGeqCofibers} for the
  cofiber sequence $A_n \to A_n/A_{n-1} \to \Sigma A_{n-1}$. The final
  statement follows by induction down from $n$.
\end{proof}

\begin{proposition}\label{cell-filts-exist} %
  If $\cat C$ is a stable $\oo$\hyp{}category equipped with a bounded
  and non\hyp{}degenerate weight structure, then every object in $\cat
  C$ admits a bounded cellular filtration.
\end{proposition}
\begin{proof}
  Suppose $X$ is an object of $\cat C$. Then even if the weight
  structure on $\cat C$ is not bounded, we can fix weight
  decompositions $A_{w\leq n} \to X$ for all $n \in \Z$. Corollary
  \ref{cor1} implies the existence of maps $A_n \to A_m$ in the
  homotopy category for $n\leq m$. These can be lifted to a coherent
  diagram $N\Z \to \cat C$ but for a bounded weight structure this is
  even simpler. In this case, $X$ has weight $-N \leq w\leq N$ for
  some $N \geq 0$. Set $A_i = X$ for $i\geq N$ and use the weight
  decomposition starting with $A_N=X$ to inductively find weight
  decompositions for $A_n$ at weight $n-1$ to get $A_{n-1} \to A_n$.
  Proposition \ref{PropCofibBound} implies that the fiber
  $A_n/A_{n-1}$ is in weight $w=n$ as desired. Each of these maps can
  be lifted from the homotopy category to $\cat C$. For $i \leq -N$,
  $A_i$ is a zero object of $\cat C$ by the non-degeneracy of the
  weight structure. The compositions of these maps and the retraction
  of $\Z$ onto $\Delta^{2N}$ as the interval $[-N,N]$ induce the
  desired functor $N\Z \to \cat C$.
\end{proof}

\subsection{Waldhausen structure on cell complexes}

We want to give $\cell^\bdd \cat C$ and $\cell^\bdd_n\cat C$
compatible Waldhausen $\oo$\hyp{}category structures. This amounts to
selecting ingressive edges. If we think of cell complexes as diagrams
in $\cat C$, an edge in $\cell^\bdd \cat C$ is a diagram
\begin{equation*}
  \begin{tikzcd}
    \cdots \ar[r, tail] & A_{i-1} \ar[d] \ar[r, tail] & A_i \ar[r, tail] \ar[d]
    & A_{i+1} \ar[d] \ar[r, tail] & \cdots \\
    \cdots \ar[r, tail] & B_{i-1} \ar[r, tail] & B_i \ar[r, tail] &
    B_{i+1} \ar[r, tail] & \cdots
  \end{tikzcd}
\end{equation*}
and we have a choice of which edges to make ingressive. Just requiring
that all vertical maps are ingressions in $\cat C$ does not imply that
the induced maps $A_j/A_i \to B_j/B_i$ are ingressions. As noted in
\cite{barwick}*{5.6} and \cite{1126}*{1.1.2}, we need a latching
condition on the diagrams: that for any $i<j$, the map from $A_j
\cup_{A_i} B_j \to B_i$ is an ingression in $\cat C$.
\begin{equation*}
  \begin{tikzcd}
    A_i \ar[r, tail] \ar[d, tail] \po & A_j \ar[d, tail] \ar[ddr, bend left, tail] \\
    B_i \ar[r, tail] \ar[drr, bend right, tail] & A_j \cup_{A_i} B_i
    \ar[dr,
    tail] \\
    & & B_j
  \end{tikzcd}
\end{equation*}
This result follows by considering the following commuting cube.
\begin{equation*}
  \begin{tikzcd}[column sep=small]
    A_i \ar[dd, tail] \ar[rr] \ar[dr, tail] && A_j \ar[dd] \ar[rr, equals] \ar[dr, tail]  && A_j \ar[dd] \ar[dr, tail]  \\
    & B_i \ar[rr, tail, crossing over] && A_j \cup_{A_i} B_i
    \ar[rr, tail, crossing over] &&
    B_j \ar[dd] \\
    * \ar[rr, tail] \ar[dr, equals] && A_j/A_i \ar[rr, equals] \ar[dr,
    equals] &&
    A_j/A_i \ar[dr, tail] \\
    & * \ar[rr, tail] \ar[uu, crossing over, leftarrow] && A_j/A_i
    \ar[rr, tail, dotted] \ar[uu, crossing over, leftarrow] && B_j/B_i
  \end{tikzcd}
\end{equation*}
If $A_j\cup_{A_i} B_i \cof B_j$ is ingressive then so is the dotted edge.

We require further that the cofiber of an ingressive map $A\cof B$ in
$\cell_n^\bdd \cat C$ is also a degree-$n$ bounded cell complex in
$\cat C$. The cofiber is computed levelwise and we want, in
particular, for the cofiber of $B_{i-1}/A_{i-1} \cof B_i/A_i$ to have
weight $w=i$.  This cofiber is identified with the cofiber of the map
$A_i \cup_{A_{i-1}} B_{i-1} \cof B_i$ which we will require to have
weight $w=i$. More generally, an ingression $A\to B$ will be levelwise
ingressions $A_i\cof B_i$ with the map $A_j \cup_{A_i} B_i \cof B_j$
an ingression in $\cat C$ with cofiber of weight $j+1 \leq w \leq i$.

Following Barwick (see \cite{barwick}*{5.6} and the following
discussion), it is easier to define the Waldhausen structure on
$\cell_n^\bdd\cat C$ as follows.
\begin{definition}
  For $i\leq j$, write $e_{i,j} : \Delta^1 \to N\Z$ for the map
  hitting $i$ and $j$.  Let $(\cell^\bdd_n \cat C)_\dagger$ be the
  smallest subcategory spanned by the edges
  $f:\Delta^1 \to \cell^\bdd_n\cat C$, which we will write $A\to B$,
  for which the square $e_{i,j}^*f: \Delta^1 \x \Delta^1 \to \cat C$
  is either of the form
  \begin{equation*}
    \begin{tikzcd}
     A_i \ar[r, tail] \ar[d, tail] \po & A_j \ar[d, tail] \\
     B_i \ar[r, tail] & B_j
    \end{tikzcd}
  \end{equation*}
  where all the edges are ingressive and the square is a pushout
  square in $\cat C$, or of the form
  \begin{equation*}
    \begin{tikzcd}
      A_i \ar[r, tail] \ar[d, "\sim"] & A_j \ar[d, tail] \\
      B_i \ar[r, tail] & B_j
    \end{tikzcd}
  \end{equation*}
  with the cofiber of the map $A_j\cof B_j$ having weight
  $i+1 \leq w \leq j$, where here the left arrow is an equivalence in
  $\cat C$ and the right is an ingression.

  We write $(\cell_n^\bdd \cat C)_\dagger$ for the subcategory of
  ingressions in $\cell_n^\bdd \cat C$ and
  $(\cell^\bdd \cat C)_\dagger$ for the ingressions in
  $\cell^\bdd \cat C$.
\end{definition}
\begin{lemma}\label{cellingressions}
  An edge $f$ of $\cell_n^\bdd \cat C$ is ingressive if and only if for
  any $e_{i,j}:\Delta^1 \to N\Z$ and any diagram $X$ from the pair
  $\oo$-category
  \begin{equation*}
    \begin{tikzcd}
      0 \ar[r, tail] \ar[d] & 1 \ar[d] \ar[ddr, bend left] \\
      2 \ar[r, tail] \ar[drr, bend right, tail] & \oo' \ar[dr] \\
      && \infty
    \end{tikzcd}
  \end{equation*}
  where $X|_{0,1,2,\oo}$ is a pushout square, the marked edges are
  ingressions, and
  $X|_{0,1,2,\oo'} = e_{i,j}^*f : (\Delta^1)^\sharp \x
  (\Delta^1)^\sharp \to \cat C$,
  then $X(\oo) \to X(\oo')$ is an ingression in $\cat C$ with cofiber
  in $\cat C_{i +1 \leq w \leq j}$.
\end{lemma}
\begin{proof}
  Ingressions on $\cell_n^\bdd \cat C$ are defined by their
  restrictions along the $e_{i,j}$. The resulting types of squares in
  the definition all admit the desired property: in the first case we
  merely note that zero objects are in every weight and the second
  case the requirement on the vertical map in the square is precisely
  what is required for the map from the pushout. Hence all ingressive
  maps satisfy the lemma.

  For the converse, we can factor any map satisfying the condition
  into a composite of maps satisfying the definition. Say $f:A\to B$
  satisfies the lemma. There is some $k$ so that $\cotr_k A\simeq A$
  and $\cotr_k B\simeq B$. Then $A_k$ and $B_k$ are both zero objects
  in $\cat C$, so the map $\tr_k f: \tr_k A \to \tr_k B$ is an
  equivalence and hence is ingressive in $\cell_n^\bdd \cat C$. Form
  the pushout $\tr_k B \cup_{\tr_k A} A$ levelwise. The map from $A$
  to this pushout is directly an ingression.
  \begin{equation*}
    \begin{tikzcd}
      A_k \ar[r, tail] \ar[d, "\simeq"] \po & A_{k+1} \ar[r, tail] \ar[d, tail] \po & A_{k+2} \ar[r, tail] \ar[d, tail] \po & \cdots \\
      B_k \ar[r, tail] & B_k \cup_{A_k} A_{k+1} \ar[r, tail] & B_k
      \cup_{A_k} A_{k+2} \ar[r, tail] & \cdots
    \end{tikzcd}
  \end{equation*}
  All the squares in this diagram are pushouts by \cite{HTT}*{4.4.2.1},
  and hence the map directly satisfies the definition of
  ingressive. $f$ induces a map
  $\tr_k B \cup_{\tr_k A} A \to \tr_{k+1} B \cup_{\tr_{k+1} A} A$
  which we write below as the second row of maps.
  \begin{equation*}
    \begin{tikzcd}
      A_k \ar[r, tail] \ar[d, "\simeq"] \po & A_{k+1} \ar[r, tail] \ar[d, tail] \po & A_{k+2} \ar[r, tail] \ar[d, tail] \po & \cdots \\
      B_k \ar[r, tail] \ar[d, "\simeq"] & B_k \cup_{A_k} A_{k+1}
      \ar[r, tail] \ar[d, dotted, tail] \po & B_k \cup_{A_k} A_{k+2}
      \ar[r, tail] \ar[d, tail] \po & \cdots \\
      B_k \ar[r, tail] & B_{k+1} \ar[r, tail] & B_{k+1} \cup_{A_{k+1}}
      A_{k+2} \ar[r, tail] & \cdots
    \end{tikzcd}
  \end{equation*}
  Repeated application of \cite{HTT}*{4.4.2.1} demonstrates that the
  marked squares are pushouts, and application of the hypothesis shows
  that the dotted arrow induced by $f$ is ingressive in $\cat C$ and
  has a cofiber of the appropriate weight.  Induction now factors $f$
  as a composite of ingressions
  $A \to \tr_k B \cup_{\tr_k} A \to \cdots \to \tr_n B \cup_{\tr_n A}
  A \simeq B$ as $\tr_n B\simeq B$ and $\tr_n A \simeq A$.
\end{proof}

\begin{proposition}
  The pair $\oo$-category
  $(\cell^\bdd \cat C, (\cell^\bdd \cat C)_\dagger)$ of bounded cell
  complexes and the pair $\oo$-category
  $(\cell_n^\bdd \cat C, (\cell_n^\bdd \cat C)_\dagger)$ of bounded
  and $n$-truncated cell complexes each form a Waldhausen
  $\oo$-category.
\end{proposition}
\begin{proof}
  As $\cell^\bdd \cat C$ is the colimit of the subcategories
  $\cell_n^\bdd \cat C$ and the same is true for the ingressions, it
  suffices to check that
  the pair $(\cell_n^\bdd\cat C, (\cell_n^\bdd \cat C)_\dagger)$ form a
  Waldhausen $\oo$-category.

  As zero objects in $\cat C$ are in all weights, the constant diagram
  at a zero object in $\cat C$ forms a zero object in
  $\cell_n^\bdd \cat C$. For any $A$ in $\cell_n^\bdd\cat C$, the
  cofiber of the map $A_i\to A_j$ has weight $i+1 \leq w \leq j$, so
  the map $0\to A$ satisfies the lemma above and hence is ingressive.

  Now suppose $A\cof B$ is an ingression in $\cell_n^\bdd \cat C$ and
  $A\to C$ is an arbitrary map. The pushout in diagrams $B \cup_A C$
  is formed levelwise with $(B\cup_A C)_i = B_i\cup_{A_i} C_i$. The
  maps $B_i \cup_{A_i} C_i \to B_j \cup_{A_j} C_j$ are ingressions in
  $\cat C$. As equivalences are checked levelwise, $\tr_n ( B\cup_A C)
  \simeq \tr_nB \cup_{\tr_n A} \tr_n C$ is equivalent to $B\cup_A
  C$. Since the same holds for cotrunctation, if $B \cup_A C$ is a
  cellular complex it will lie in $\cell_n^\bdd \cat C$. It remains to
  show that $B\cup_A C$ is a cellular complex and the map $C\to
  B\cup_A C$ is an ingression in $\cell_n^\bdd \cat C$. Both amount to
  checking that certain

  Write $D$ for the pushout $B\cup_A C$. As pushouts commute,
  $D_j/D_i \simeq (B_j/B_i)\cup_{A_j/A_i} (C_j/C_i)$.
  \begin{equation*}
    \begin{tikzcd}
      A_j/A_i \ar[r, tail] \ar[d] \po & B_j/B_i \ar[d] \\
      C_j/C_i \ar[r, tail] & D_j/D_i
    \end{tikzcd}
  \end{equation*}
  Hence, the cofiber of the top and bottom map are equivalent in
  $h\cat C$, so $$(D_j/D_i) / (C_j/C_i) \simeq (B_j/B_i) / (A_j/A_i)$$ 
  which is equivalent to $B_j / (A_j\cup_{A_i} B_i)$ by commuting
  pushouts again. The latter is in weight $i+1 \leq w \leq j$ by
  assumption on the map $A\to B$. We note that since weights (both
  $\cat C_{w\geq i+1}$ and $\cat C_{w\leq j}$) are closed under
  extension by definition, the cofiber sequence
  $C_j/C_i \cof D_j/D_i \to (D_j/D_i)/(C_j/C_i)$ now shows that
  $D_j/D_i$ also has weight $i+1\leq w \leq j$ as desired. Hence, $D$
  lies in $\cell^\bdd_n\cat C$ and the map $C\to D$ is an
  ingression. We note that this analysis did not require the
  particular model of $D$ as the levelwise pushout, so we also
  conclude that any pushout of an ingression in $\cell_n^\bdd \cat C$
  is also an ingression.
\end{proof}

\subsection{Localizing cell complexes}

By construction, the equivalences in $\cell \cat C$ are those maps
which induce equivalences in $\cat C$ degreewise. This is too rigid:
two cell complexes are only equivalent if all the $n$-skeleta are
equivalent. We would like to make all cell complexes for a single
object in $\cat C$ equivalent to each other.

We regard the functor $\colim_\Z$ as taking a cell complex to the
object in $\cat C$ it models. We are primarily interested in bounded
complexes, which are filtered by the subcategories
$\cell_n^\bdd \cat C$ of complexes $A$ that precisely carry the data
of a cellular filtration for $A_\oo \simeq A_n$.

Denote by $v\cell \cat C$ (or $v\cell^\bdd\cat C$ or
$v\cell_n^\bdd \cat C$) the subcategory of $\cell\cat C$ which
$\colim_\Z$ takes to equivalences in $\cat C$. We'd like to localize
the bounded cell complexes at $v\cell^\bdd \cat C$ which will require
Barwick's labeled Waldhausen $\oo$-categories \cite{barwick}*{2.9}.
The virtual Waldhausen $\oo$-category given by $\cell^\bdd \cat C$
labeled by $v\cell^\bdd\cat C$ will be our surrogate for $\cat C$.

In the proof of the main theorem, we compare the $K$-theory of the
localization
\begin{equation*}
  (v\cell^\bdd \cat C)\inv \cell^\bdd \cat C
\end{equation*}
to that of the unlocalized cell complexes $\cell^\bdd \cat C$. The
following result compares this directly to $K(\cat C)$.
\begin{proposition}\label{localization_prop}
  The $K$-theory of the localization is equivalent to that of $\cat C$.
  \begin{equation*}
    K((v\cell^\bdd \cat C)\inv \cell^\bdd \cat C) \simeq K(\cat C)
  \end{equation*}
\end{proposition}
\begin{proof}
  Not every map in $(v\cell^\bdd \cat C)\inv \cell^\bdd \cat C$ is
  necessarily ingressive. However, we can apply Fiore's approximation
  theorem \cite{fiore} in this situation. Write $F$ for the functor
  induced by $\colim_\Z$ from the localization $(v\cell^\bdd \cat
  C)\inv \cell^\bdd \cat C$ to $\cat C$. By proposition
  \ref{cell-filts-exist}, $F$ is essentially surjective. By
  construction of $v\cell^\bdd \cat C$, $F$ reflects equivalences in
  $\cat C$. Finally, any diagram indexed by a finite poset in the
  maximal Kan subcategory of $(v\cell^\bdd \cat C )\inv \cell^\bdd
  \cat C$ admits a colimit in $(v\cell^\bdd \cat C )\inv \cell^\bdd
  \cat C$ which can be constructed as the colimit of a diagram indexed
  on a finite poset in $\cell^\bdd \cat C$ where all maps induce
  equivalences on $\colim_\Z$. These colimits are constructed
  levelwise and since the poset is finite there is some $N$ where all
  terms achieve $\colim_\Z$, so we directly see that $\colim_\Z$
  preserves those colimits. Hence $F$ satisfies the hypotheses of
  \cite{fiore}*{4.5} to show that it induces an equivalence of
  homotopy categories on the subcategories of ingressions. The
  approximation theorem \cite{fiore}*{4.10} now implies that $F$
  induces an equivalence on $K$-theory.
\end{proof}

\section{The main theorem}\label{spheretheorem}
This section provides the proof of theorem \ref{spherethm} and a
description of $K(\cat C_{\heart w})$ analogous to the
``plus-equals-$Q$'' theorem. Throughout, we let $\cat C$ denote a
fixed stable $\oo$-category equipped with a bounded weight structure
$w$.

\begin{theorem}\label{spherethm}
  If $\cat C$ is a stable $\oo$-category equipped with a bounded
  non-degenerate weight structure $w$, then the inclusion of the heart
  of the weight structure
  $\cat C_{\heart w}\hookrightarrow \cat C_{\heart}$ induces an
  equivalence on algebraic $K$-theory
  \begin{equation*}
    K(\cat C) \simeq K(\cat C_{\heart w}).
  \end{equation*}
\end{theorem}

Recall that $\cat C_{\heart w}$ is given a Waldhausen $\oo$-category
structure where ingressions are precisely those maps admitting
cofibers in $\cat C_{\heart w}$.

The proof for the theorem studies bounded cellular filtrations of
objects in $\cat C$ with respect to the weight structure $w$. Call
this category $\cell^\bdd(\cat C)$ and let $\cell^\bdd_\triv(\cat C)$
denote the subcategory of cellular filtrations of zero objects in $\cat
C$. The localization theorem gives a fiber sequence
\begin{equation*}
  \begin{tikzcd}
    K(\cell^\bdd_\triv(\cat C)) \ar[r] & K(\cell^\bdd(\cat C)) \ar[r]
    & K(\cat C)
  \end{tikzcd}
\end{equation*}
and additivity identifies the left and middle terms of this sequence
as infinite products of $K(\cat C_{\heart w})$. By filtering the
category of cellular filtrations, we show that the left product has
one fewer copy of $K(\cat C_{\heart w})$
\begin{equation*}
  \begin{tikzcd}
    \displaystyle \prod_{\textrm{one fewer}} K(\cat C_{\heart w}) \ar[r] & \prod
    K(\cat C_{\heart w}) \ar[r] & K(\cat C)
  \end{tikzcd}
\end{equation*}
which implies $K(\cat C_{\heart w})\simeq K(\cat C)$ as desired.

Bondarko proves this theorem on $K_0$-groups by considering bounded
weight structures on triangulated categories
\cite{MR2746283}*{5.3.1}. We independently reproduce his result by
passing to the underlying stable $\oo$-category, applying our theorem,
and taking $\pi_0$.
\begin{corollary}[\cf \cite{MR2746283}*{5.3.1}]\label{bondarko_k0}
  If $\cat T$ is a triangulated category equipped with a
  non\hyp{}degenerate bounded weight structure $w$, then the inclusion
  of the heart $\cat T_{\heart w}$ into $\cat T$ induces an
  equivalence on $K$-theory,
  $K_0(\cat T) \simeq K_0(\cat T_{\heart w})$.
\end{corollary}

At the end of the section, we study the $K$-theory of the heart of a
weight structure. All ingressions in $\cat C_{\heart w}$ split in the
homotopy category, so the $K$-theory admits a description in the style
of Quillen's plus construction for $K$-theory.

\subsection{Proving the main theorem}

Using the technology of cellular filtrations constructed in section
\ref{cellfiltrations}, we study the $K$-theory of $\cat C$ through the
labeled pair Waldhausen $\oo$-category
$(\cell^\bdd \cat C, v\cell^\bdd \cat C )$. %
We use the localization theorem to relate the algebraic $K$-theory of
this pair to the algebraic $K$-theory of $\cell^\bdd \cat C$.
\begin{theorem}[{\cite{barwick}*{9.24}}] %
  Suppose $(\cat A,w \cat A)$ is a labeled Waldhausen $\oo$-category
  that has enough cofibrations. Suppose $\phi:\Wald_\oo \to E$ is an
  additive theory with left derived functor $\Phi$. Then the inclusion
  $i:\cat A^w \to \cat A$ and the morphism of virtual Waldhausen
  $\oo$-categories $e:\cat A \to \cat B(\cat A, w\cat A)$ give rise to
  a fiber sequence
  \begin{equation*}
    \begin{tikzcd}
      \phi(\cat A^w) \ar[r] \ar[d] & \phi(\cat A) \ar[d] \\
      * \ar[r] & \Phi(\cat B(\cat A,w\cat A)).
    \end{tikzcd}
  \end{equation*}
\end{theorem}
Here, $\cat B(\cat A,w\cat A)$ is the virtual Waldhausen
$\oo$-category corresponding to the pair $(\cat A, w\cat A)$. For our
result $\phi$ will be $K$-theory, and the derived $K$-theory on the
virtual Waldhausen $\oo$-category of the pair
$K(\cat B(\cat A, w\cat A))$ will be written simply as the $K$-theory
of the pair $K(\cat A,w\cat A)$.  In this theorem, $\cat A^w$ denotes
the full subcategory of $w$-acyclic objects in $\cat A$. In our
setting, the $v$-acyclic objects of $\cell^\bdd\cat C$ are classified
by the following computation.
\begin{proposition}\label{prop14}
  If $A$ is a $v$-acyclic object of $\cell^\bdd \cat C$ then
  $A_n$ has weight $w=n$ in $\cat C$ for all $n$.
\end{proposition}
\begin{lemma}
  If
  \begin{equation*}
    \begin{tikzcd}
      A\ar[r] & B \ar[r] & C
    \end{tikzcd}
  \end{equation*}
  is a cofiber sequence in $\cat C$ with $B \in \cat C_{w\geq m}$ and $C\in
  \cat C_{w\geq n+1}$ and $m\geq n+1$ then $A\in \cat C_{w \geq n}$.
\end{lemma}
\begin{proof}
  For any $X \in \cat C_{w\leq n-1}$ we obtain the following exact sequence
  of mapping spaces in $h\cat C$.
  \begin{equation*}
    \begin{tikzcd}
      hC(X,\Sigma^{-1} C) \ar[r] & hC(X,A) \ar[r] & hC(X,B)
    \end{tikzcd}
  \end{equation*}
  Now since $C\in \cat C_{w\geq n+1}$, $\Sigma^{-1} C$ lives in $\cat C_{w\geq n}$ so
  the left mapping space is trivial. The same is true of the right
  mapping space since $B\in \cat C_{w\geq m}$ and $m \geq n+1 \geq n$. We
  conclude that $A$ is in $\cat C_{w\geq n}$.
\end{proof}
\begin{proof}[Proof of proposition \ref{prop14}]
  The proof follows from induction down from the finite stage where
  $A$ achieves its colimit, a zero object. Say that
  $A\in \cell_n^\bdd \cat C$ so that we have an equivalence
  $* \simeq A_n$, letting us conclude that $A_n$ has weight
  $w=n$. Induction using the lemma above when $m=k+1$ lets us conclude
  that $A_k$ has weight $w \geq k$ for all $k$. A similar induction
  from below demonstrates that $A_n \in \cat C_{w\leq n}$ for any
  sequence in $A$. This completes the proof.
\end{proof}

\begin{remark}
  As $\cat C_{w=n}$ consists of the ``pure'' objects in the weight
  structure, $(\cell^\bdd \cat C)^v$ can be thought of as ``finite
  Kozsul resolutions'' between zero objects in $\cat C$. We note that
  $\cat C_{\heart w} = \cat C_{w=0}$ is equivalent (via suspension) to
  $\cat C_{w=n}$ for any $n$.
\end{remark}

To use the localization theorem, we must check that the marked cell
filtrations satisfy the technical hypothesis of having enough
cofibrations.
\begin{proposition}\label{prop_enough_cofs}
  When $\cat C$ is a stable $\oo$-category equipped with the maximal
  Waldhausen $\oo$-category structure, the labeled Waldhausen
  $\oo$-category $(\cell^\bdd \cat C, v\cell^\bdd \cat C)$ has enough
  cofibrations.
\end{proposition}
\begin{proof}
  By \cite{barwick}*{9.22}, it is sufficient to construct a functorial
  mapping cylinder $M$ on arrows
  $$M:\Fun(\Delta^1, \cell^\bdd \cat C) \to \Fun(\Delta^1, \cell^\bdd
  \cat C)$$
  that produces ingressive arrows, preserves
  $v$-equivalences, and comes with a natural transformation
  $\eta: \id \to M$ which is an objectwise labeled by
  $v$-equivalences.
  
  We construct $M$ for arrows in $\Fun(N\Z, \cat C)$ and show that it
  produces arrows in $\cell^\bdd \cat C$. Let $\sh_{-1}$ denote the
  functor induced on $\Fun(N\Z,\cat C)$ by the map $z\mapsto z-1$ on
  $\Z$. Note that $\sh_{-1}(A)_i = A_{i-1}$. Also note that there is a
  natural transformation $\sh_{-1} \to \id$ whose levelwise maps are
  just the structure maps, \ie $A_{i-1}\to A_i$.

  For $f:A\to B$ an arrow, $Mf$ will be defined to be the pushout
  \begin{equation*}
    \begin{tikzcd}
      \sh_{-1}(A) \ar[r, "\sh_{-1}(f)"] \po \ar[d] & \sh_{-1}(B) \ar[r] & B \ar[d] \\
      A \ar[rr] & {} & Mf
    \end{tikzcd}
  \end{equation*}
  and since colimits are computed levelwise on diagram categories, we
  observe that
  \begin{equation*}
    (Mf)_i \simeq B_i \cup_{A_{i-1}} A_i.
  \end{equation*}
  We regard $M(f)$ as the arrow $A \to Mf$ and will write $Mf$ only
  for the target object.  We note that the construction is functorial
  on the arrow category for $\Fun(N\Z, \cat C)$.
  
  First we check that $Mf$ is an endofunctor for arrows in
  $\cell^\bdd(\cat C)$. By commuting pushouts, we find that
  $(Mf)_{i+1}/(Mf)_{i}$ is the (homotopy) pushout
  \begin{equation*}
    \begin{tikzcd}
      A_{i}/A_{i-1} \ar[r] \ar[d] \po & B_{i+1}/B_i \ar[d] \\
      A_{i+1}/A_i \ar[r] & (Mf)_{i+1}/(Mf)_i
    \end{tikzcd}
  \end{equation*}
  but due to the weights of the objects, the top and left map are both
  $0$, so $(Mf)_{i+1}/(Mf)_i$ splits up as a wedge sum
  \begin{equation*}
    (Mf)_{i+1}/(Mf)_i \simeq A_{i+1}/A_i \vee B_{i+1}/ B_i \vee \Sigma(A_i / A_{i-1})
  \end{equation*}
  which has weight $w=i+1$ as desired. Since $A$ and $B$ are bounded,
  so will $Mf$. Hence $M(f)$ is an arrow in $\cell^\bdd(\cat C)$ if
  $f$ is as well.

  Next we check that $M(f)$ is a cofibration in $\cell^\bdd(\cat
  C)$. It suffices to check the latching condition for the map $g$.
  \begin{equation*}
    \begin{tikzcd}
      A_i \ar[r] \ar[d] \po & A_{i+1} \ar[d] \ar[ddr, bend left] \\
      (Mf)_i \ar[r] \ar[drr, bend right] & P \ar[dr, "g"] \\
      && (Mf)_{i+1}
    \end{tikzcd}
  \end{equation*}
  From the construction of $Mf$, we have the following pushout squares
  \begin{equation*}
    \begin{tikzcd}
      A_{i-1} \ar[r] \ar[d] \po & A_i \ar[d] \ar[r] \po & A_{i+1} \ar[d] \\
      B_i \ar[r] & (Mf)_i \ar[r] & P
    \end{tikzcd}
  \end{equation*}
  and the outer square is also a pushout square by
  \cite{HTT}*{4.4.2.1}. We map the outer pushout square to
  \begin{equation*}
    \begin{tikzcd}
      A_i \ar[r] \ar[d] \po & A_{i+1} \ar[d] \\
      B_{i+1} \ar[r] & M_{i+1}
    \end{tikzcd}
  \end{equation*}
  by the obvious maps. We take the cofiber of this map of pushout
  squares which, since colimits commute, is also a pushout square.
  \begin{equation*}
    \begin{tikzcd}
      A_i / A_{i-1} \ar[r] \ar[d] \po & * \ar[d] \\
      B_{i+1}/B_{i} \ar[r] & M_{i+1} / P
    \end{tikzcd}
  \end{equation*}
  The weights of $A_i/A_{i-1}$ and $B_{i+1}/B_{i}$ imply that
  $M_{i+1} \simeq B_{i+1}/B_i \vee \Sigma(A_i / A_{i-1})$ which lives
  in weight $w=i+1$ as desired.

  Now we claim that $M$ is a mapping cylinder for $f$ with respect to
  the $v$-equivalences. That is, we show that the natural map
  $B \to Mf$ is a $v$-equivalence. Denote this map $\phi:B\to Mf$. The
  maps $f:A\to B$ and $\id:B\to B$ also induce a map $\psi:Mf \to
  B$. Levelwise, these maps appear as
  \begin{equation*}
    \begin{tikzcd}
      A_{i-1} \ar[r] \ar[d] \po & B_i \ar[d, "\phi"] \ar[ddr, bend left, "\id"] \\
      A_i \ar[r] \ar[drr, "f_i", bend right] & (Mf)_i \ar[dr, "\psi"] \\
      & & B_i
    \end{tikzcd}
  \end{equation*}
  and $\psi\comp \phi \simeq \id$. $\psi$ is a right inverse for
  $\phi$ as well if $f_i$ factors the map from $A_i\to (Mf)_i$ through
  $B_i$ via $\phi$. This will be satisfied once $A$ and $B$ both
  achieve their limits. Equivalently, the vertical cofibers in the square
  \begin{equation*}
    \begin{tikzcd}
      A_{i-1} \ar[r] \ar[d] \po & B_i \ar[d] \\
      A_i \ar[r] & (Mf)_i
    \end{tikzcd}
  \end{equation*}
  must be equivalent. The left is $A_i/A_{i-1}$ and the right is
  $(Mf)_i/B_i$. Once $A$ no longer has any cells, we conclude that
  $\phi_i :B_i\to (Mf)_i$ is an equivalence. Hence, $\phi$ is a
  $v$-equivalence. Since $\phi$ induces the desired natural
  transformation $\id \to M$, we conclude that
  the pair $(\cell^\bdd \cat C, v\cell^\bdd\cat C)$ has enough cofibrations.
\end{proof}

We apply Barwick's localization theorem to the algebraic $K$-theory of
the labeled Waldhausen $\oo$-category $(\cell^\bdd \cat C, v\cell^\bdd
\cat C)$ to produce the following pushout diagram.
\begin{equation*}
  \begin{tikzcd}
    K((\cell^\bdd \cat C)^v) \ar[r] \ar[d] \po & K(\cell^\bdd \cat C) \ar[d] \\
    * \ar[r] & K(\cell^\bdd \cat C, v\cell^\bdd \cat C)
  \end{tikzcd}
\end{equation*}
The top map is induced by the inclusion of the $v$-acyclics into
$\cell^\bdd \cat C$. The proof will proceed by factoring this map on
$K$-theory and analyzing the resulting diagram.

Integral to our argument are two functorial ways of embedding
$\cat C_{\heart w}$ into cell complexes. On one hand, we can include
an $n$-spherical object $a_n$ as the cell filtration concentrated in
degree $n$, where we attach $a_n$ to a zero object and then keep the
filtration constant. This functor essentially includes $a_n$ as the
filtration $0 \to a_n$ living in $\cell_n^\bdd\cat C$. We will abuse
notation and write the resulting complex simply as $a_n$ when its
weight is clear. We can also include $a_n$ into the $v$-acyclic cell
filtrations by including it as the filtration $0 \to a_n \to 0'$ where
$a_n$ is attached to $0$ and then immediately killed at the next
level. This will be referred to as the \emph{cone} on $a_n$, written
$\cone(a_n)$, and includes $a_n$ into $(\cell_{n+1}^\bdd\cat C)^v$.
Coherent functoriality of these maps is ensured by the following
construction.

By \cite{HTT}*{3.2.2}, the source map
$s: Fun(\Delta^1, \cat C_{w=i}) \to \cat C_{w=i}$ is a cartesian
fibration and if we denote the full subcategory of zero objects in
$\cat C$ by $\cat C_0$, pulling back $s$ over the inclusion of
$\cat C_0 \to \cat C_{w=i}$ yields a cartesian fibration
$\cat C_0 \x_{s} \Fun(\Delta^1, \cat C_{w=i})$ which at a zero object
$*$ classifies the $\oo$-category of arrows $* \to X$ in $\cat C$
where $X$ has weight $w=i$. Let $p_i:N\Z\to \Delta^1$ be defined by
$p_i(j)= 0$ for $j<i$ and $p_i(j) = 1$ for $j\geq i$. Pullback along
$p_i$ and inclusion into $\cat C$ induces a map
$\cat C_0 \x_s \Fun(\Delta^1,\cat C_{w=i}) \to \Fun(N\Z, \cat C)$.
This carries an arrow $*\to X$ with target in weight $w=i$ to a
filtration that is evidently cellular and bounded as the level
quotients are all zero objects except for at degree $i$ where it is
equivalent to $X$. We will denote the resulting functor from
$\cat C_{w=i} \to \cell_n^\bdd\cat C$ by $c_i$, or, when clear, with
no decoration, as this is the natural way to include $\cat C_{w=i}$
into $\cell_n^\bdd \cat C$.

To produce the cone functor, we use that the maps $ev_0$ and
$ev_1 : \Fun(\Delta^2, \cat C_{w=i}) \to \cat C_{w=i}$ are also
cartesian fibrations. We pull back the map $(ev_0, ev_2)$ from
$\Fun(\Delta^2,\cat C_{w=i}) \to \cat C_{w=i} \x \cat C_{w=i}$ along
the inclusion of $\cat C_0 \x \cat C_0$. Above a pair of zero objects
this classifies composable pairs of arrows $0\to X \to 0'$ with $X$ in
$\cat C_{w=i}$. We pull back along the map $N\Z \to \Delta^2$ which
collapses $\Z$ onto the interval $[i-1, i+1]$ as above. Observe that
the resulting cell filtration lives in $\cell_{i+1}^\bdd \cat C$ and
$\colim_\Z$ takes it to a zero object. Denote the corresponding
functor $\cat C_{w=i} \to (\cell_{i+1}^\bdd \cat C)^v$ by $\cone$.

$(n-1)$-truncation $\tr_{n-1}$ is an exact functor $\cell^\bdd\cat C$
to $\cell_{n-1}^\bdd\cat C$ by lemma
\ref{cellingressions}. Suppressing the inclusion
$\cell_{n-1}^\bdd \cat C \to \cell_n^\bdd \cat C$ as well as the
restriction to the subcategory, we consider $\tr_{n-1}$ as an
endofunctor of $\cell_n^\bdd \cat C$. The truncation $\tr_{n-1}$ comes
with a natural transformation to $\id$ by construction. We take the
(homotopy) cofiber of this map of endofunctors $\tr_{n-1}\to \id$ in
the arrow $\oo$-category of functors. Denote the resulting endofunctor
of $\cell_n^\bdd \cat C$ by $q_{n}$.

The functor $q_{n}$ is constructed to be a homotopy-coherent model for
the top level quotient $A \mapsto A_n/A_{n-1}$. In particular, since
$\colim_\Z\tr_{n-1}A$ is equivalent to $A_{n-1}$ and $\colim_\Z A$ is
equivalent to $A_n$, we observe that
$\colim_\Z q_{n-1} A \simeq A_n/A_{n-1}$ in the homotopy category of
$C$. Alternatively, the cofiber is computed pointwise, and we observe
that the image of $q_n$ is equivalent to $C_{w=n}$ included as
constant cell complexes concentrated in degree $n$.
\begin{proposition}
  The map
  \begin{equation*}
    \begin{tikzcd}
      K(\cell_{n-1}^\bdd \cat C) \x K(\cat C_{w=n}) \ar[r, "\vee"] &
      K(\cell_n^\bdd \cat C)
    \end{tikzcd}
  \end{equation*}
  is an equivalence, where we include $\cat C_{w=n}$ as constant cell
  complexes concentrated in degree $n$.
\end{proposition}
\begin{proof}
  The pushout square of endofunctors
  \begin{equation*}
    \begin{tikzcd}
      \tr_{n-1} \ar[r] \ar[d] \po & \id \ar[d] \\ 0 \ar[r] & q_n
    \end{tikzcd}
  \end{equation*}
  implies by \cite{barwick}*{7.4.(5)} %
  that $K(\tr_{n-1} \oplus q_n)$ is
  equivalent to $K(\id)$ on $K(\cell_n^\bdd \cat C)$. Hence the
  inclusion $(\cell_n^\bdd \cat C)^v \to \cell_n^\bdd \cat C$ factors on
  $K$-theory as
  \begin{equation*}
    \begin{tikzcd}
      K((\cell_{n-1}^\bdd \cat C) \x
      \im q_{n}) \simeq K((\cell_{n-1}^\bdd \cat C) \x \cat C_{w=n})
      \ar[dr, "\vee"]\\
      K( (\cell_n^\bdd \cat C)^v) \ar[u, "\tr_{n-1}\oplus q_n"]
      \ar[r, "\incl"'] & K(\cell_{n}\cat C)
    \end{tikzcd}
  \end{equation*}
  where the middle term is determined by the images of $\tr_{n-1}$ and
  $q_n$, which on $(\cell_n^\bdd\cat C )^v$ are an $(n-1)$-truncated
  cell complex and a cell complex concentrated in degree $n$, which is
  the essential image of the weight-$n$-spheres under the
  constant-cell-complex functor
  $\cat C_{w=n}\to \cell_{n}^\bdd \cat C$. The vertical map $\vee$ is
  the wedge of the inclusion and the constant-cell-filtration functor
  from $\cat C_{w=n}$.

  Furthermore, the wedge product map is a right inverse to
  $\tr_{n-1} \oplus q_n$ before taking $K$-theory, so we conclude that
  all the maps are equivalences on $K$-theory and hence the wedging
  map is an equivalence on $K$-theory.
\end{proof}
We can induct down on $n$ to arrive at the following result.
\begin{corollary}\label{cell_prod_heart}
  $K(\cell_n^\bdd \cat C)$ is equivalent to
  $\prod_{i\leq n} K(\cat C_{w=i})$ under the map induced by
  $q = q_n \oplus q_{n-1}\tr_{n-1} \oplus q_{n-2}\tr_{n-2}\oplus
  \cdots$.
\end{corollary}

Now we can factor the pushout square from the fibration theorem into
two pushout squares
\begin{equation*}
  \begin{tikzcd}
    K( (\cell_n^\bdd \cat C)^v) \ar[r] \ar[d]
    \po & K( \cell_{n-1}^\bdd \cat C ) \x K(\cat C_{w=n}) \ar[d] \\
    * \ar[r] & P 
  \end{tikzcd}
\end{equation*}
and
\begin{equation*}
  \begin{tikzcd}
    K( \cell_{n-1}^\bdd \cat C ) \x K(\cat C_{w=n}) \ar[r, "\simeq"]
    \ar[d] \po &
    K( \cell_n^\bdd \cat C) \ar[d] \\
    P \ar[r] & K(\cell_n^\bdd \cat C, v\cell_n^\bdd \cat C)
  \end{tikzcd}
\end{equation*}
and reduce the proof to two outstanding claims.
\begin{claim}\label{claim2}
  $K(\cell_n^\bdd\cat C,v\cell_n^\bdd \cat C) \simeq K( (v\cell_n^\bdd
  \cat C)\inv \cell_n^\bdd \cat C)$.
\end{claim}
\begin{claim}\label{claim0}
  The pushout $P$ is equivalent to
  $K(\cat C_{w=n}) \simeq K(\cat C_{\heart w})$.
\end{claim}
As pushouts of equivalences are equivalences, the bottom map in the
right square is an equivalence. Hence, these claims and proposition
\ref{localization_prop} let us conclude that $K(\cat C_{\heart w})
\simeq K(\cat C)$ after passing to the colimit over $n$. The
equivalence is induced by the inclusion of $\cat C_{w=n}$ into
$\cell_n^\bdd \cat C$ as constant cellular filtrations at degree
$n$. Under localization at the equivalences $v\cell_n^\bdd\cat C$,
this corresponds to the inclusion map $\cat C_{\heart w}\simeq \cat
C_{w=n}\to \cat C$ as asserted in the theorem.

\begin{proof}[Proof of claim \ref{claim2}]
  \cite{barwick}*{10.15} implies that the $K$-theory space of the pair
  \begin{equation*}
    (\cell_n^\bdd \cat C, v\cell_n^\bdd \cat C)
  \end{equation*}
  is weakly homotopy
  equivalent to
  \begin{equation*}
    \loops \colim vS_\bullet( \cell_n^\bdd \cat C)
  \end{equation*}
  where the nerve direction is taken over maps in
  $v\cell_n^\bdd \cat C$. Likewise, the $K$-theory space of the
  localization is weakly equivalent to
  \begin{equation*}
    \loops \colim \iota S_\bullet ((v\cell_n^\bdd\cat C)\inv
    \cell_n^\bdd \cat C)
  \end{equation*}
  and localization induces a simplicial functor
  \begin{equation*}
    L_\bullet:S_\bullet (\cell_n^\bdd \cat C) \to S_\bullet ( (v\cell_n^\bdd \cat C)\inv \cell_n^\bdd \cat C).
  \end{equation*}
  Since localization carries $v$-labeled edges to equivalences, this
  maps the $v$-nerve fully faithfully into the $\iota$-nerve. All we
  need to check is essential surjectivity. It suffices to check
  levelwise. This will follow from general behavior of diagrams and
  localizations.
  
  Any zig-zag
  \begin{equation*}
    \begin{tikzcd}
      A & A' \ar[l, "\sim_v"] \ar[r] & B
    \end{tikzcd}
    \textrm{or}
    \begin{tikzcd}
      A \ar[r] & B & B' \ar[l, "\sim_v"]
    \end{tikzcd}
  \end{equation*}
  in the simplicial localization receives a map from a zig-zag with
  only identity arrows reversed
  \begin{equation*}
    \begin{tikzcd}
      A & A' \ar[l, "\sim_v"] \ar[r] & B \\
      A' \ar[u, "\sim_v"] \ar[r, equals] & A' \ar[u, equals] \ar[r] &
      B \ar[u,equals]
    \end{tikzcd}
    \textrm{or}
    \begin{tikzcd}
      A \ar[r] & B & B' \ar[l, "\sim_v"] \\
      A \ar[u, equals] \ar[r] & B' \ar[u, "\sim_v"] & B' \ar[l,
      equals] \ar[u,equals]
    \end{tikzcd}
  \end{equation*}
  and we note that the vertical maps are all $v$-equivalences, hence
  equivalences in the localization. The lower zig-zags are hit by the
  localization functor. Hence, any sequence of maps in
  $S_\bullet( (v \cell^\bdd_n \cat C)\inv \cell_n^\bdd \cat C)$
  receives a map from a sequence of maps in the image of $L$. We
  conclude that $L$ induces a weak equivalence of the nerves as
  desired.
\end{proof}

Claim \ref{claim0} will follow from proving that the truncation
functor induces an equivalence on $K$-theory from
$(\cell_n^\bdd \cat C)^v$ to $\cell_{n-1}^\bdd \cat C$. We will prove this
by identifying $K((\cell_n^\bdd \cat C)^v)$ with
$\prod_{i\leq n-1} K(\cat C_{w=i})$ in $K(\cell_n^\bdd \cat C)$ which
is identified with $\prod_{i\leq n} K(\cat C_{w=i})$ using corollary
\ref{cell_prod_heart}.
\begin{lemma}
  The truncation functor
  $\tr_{n-1}: (\cell^\bdd_n \cat C )^v \to \cell_{n-1}^\bdd \cat C$
  induces an equivalence on $K$-theory.
\end{lemma}
\begin{proof}
  The truncation functor maps the acyclic $n$-cell complexes fully and
  faithfully into the $(n-1)$-cell complexes by forgetting the zero
  object at degree $n$. By corollary \ref{cell_prod_heart},
  $q: \cell^\bdd_{n-1}\cat C \to \prod_{i\leq n-1} \cat C_{w=i}$
  induces an equivalence on $K$-theory whose inverse is induced by the
  map $W$ which forms the wedge sum of constant cell
  filtrations. Hence it suffices to check that $\tr_{n-1}$ is
  essentially surjective onto the image of $W$ to induce an
  equivalence of $K$-theory. This will follow by inducting down on
  cells. In particular, we have the following diagram of functors
  \begin{equation*}
    \begin{tikzcd}
      (\cell^\bdd_{n-1} \cat C )^v \ar[r, "\tr_{n-2}"] \ar[d, "\incl"]
      & \cell_{n-2}^\bdd \x \cat C_{w=n-1} \cat C \ar[r, "q", shift
      left] \ar[d, "\vee"] & \cat C_{w=n-1} \x \prod_{i\leq n-2} \cat
      C_{w=i} \ar[d, "\simeq"]
      \ar[l, "\id \x W", shift left] \\
      (\cell_n^\bdd \cat C )^v \ar[r, "\tr_{n-1}"] & \cell_{n-1}
      \ar[r, "q", shift left] & \prod_{i \leq n-1} \cat C_{w=i} \ar[l,
      "W", shift left]
    \end{tikzcd}
  \end{equation*}
  where all functors in the right square induce equivalences on
  $K$-theory. A bounded cell filtration in $\cell^\bdd_{n-1} \cat C$
  corresponds to a sequence of level quotients
  $(a_i) \in \prod_{i\leq n-1} \cat C_{w=i}$ where all but finitely
  many are zero objects. $W$ sends this to the wedge
  $\bigvee_{i\leq n-1} a_i$ where the weight of each $a_i$ indicates
  in which degree its filtration is concentrated in
  $\cell_{n-1}^\bdd \cat C$.

  If we inductively assume that $\bigvee_{i\leq n-2} a_i$ is in the
  essential image of $\tr_{n-2}$ in $\cell_{n-2}^\bdd \cat C$, then
  observe that $\cone(a_{n-1})$ is sent to $a_{n-1}$ under
  $\tr_{n-1}$, so $a_{n-1} \vee \bigvee_{i \leq n-2} a_i$ will also be
  in the essential image. Boundedness implies that this induction
  terminates in finite steps once all cells are coned off.
\end{proof}

The homotopy fibers of the vertical maps in the pushout square
\begin{equation*}
  \begin{tikzcd}
    K((\cell_n^\bdd \cat C )^v) \ar[d] \ar[r] \po &
    K(\cell_{n-1}^\bdd \cat C ) \ar[d] \\
    * \ar[r] & P
  \end{tikzcd}
\end{equation*}
must agree, so in light of the previous lemma, we conclude that
$P \simeq K(\cat C_{w=n})$. This completes the proof of claim
\ref{claim0} as well as the proof of the main theorem.

\subsection{On the $K$-theory of the heart of a weight structure}

Let $\cat C_{\heart w}$ denote the heart of a weight structure on a
stable $\oo $-category $\cat C$. In this section, we describe the
$K$-theory of $\cat C_{\heart w}$ in terms of equivalence classes of
objects and their automorphisms.

One key observation about $\cat C_{\heart w}$ allows its $K$-theory to
be simply described: cofiber sequences in $\cat C_{\heart wx}$ all split.

\begin{proposition}\label{split-heart}
  If $f:A\to B$ is an ingression in $\cat C_{\heart w}$, then $f$ splits
  in the homotopy category $h\cat C$.
\end{proposition}
\begin{proof}
  Write $A\to B\to C$ for the cofiber sequence associated to $f$,
  considered as an exact triangle in $h\cat C$. By assumption, $C$
  lies in $\cat C_{\heart w}$ as well. Lemma \ref{lem1} implies that
  there are maps $g$ and $h$ extending $\id:A\to A$ to a map of exact
  triangles.
  \begin{equation*}
    \begin{tikzcd}
      A \ar[r, "f"] \ar[d, "\id"] & B \ar[d, dotted, "g"] \ar[r] & C
      \ar[r] \ar[d, dotted, "h"] & \Sigma A \ar[d, "\id"] \\
      A \ar[r, "\id"] & A \ar[r] & 0 \ar[r] & \Sigma A
    \end{tikzcd}
  \end{equation*}
  Thus $g\comp f \simeq \id_A$ in $h\cat C$.
\end{proof}
Recall here that ingressions in $\cat C_{\heart w}$ are those maps
which are ingressions in $\cat C$ with cofiber in $\cat C_{\heart w}$.

This case was studied by Waldhausen in \cite{1126}*{\S 1.8} and his
approach applied directly. Write $\cat D$ for any pointed Waldhausen
$\oo$-category with finite coproducts where all ingressions split up
to homotopy, even though the only case of interest to us at this point
is $\cat D = \cat C_{\heart w}$.

Define $N_\bullet \cat D$ to be the nerve of $\cat D$ with respect to
the coproduct, so $N_n \cat D$ is $\cat D^n$ and the face map is
\begin{equation*}
  d_i(X_1,\ldots, X_n) =
  \begin{cases}
    (X_2,\ldots, X_n) & i=0 \\
    (X_1,\ldots, X_i \vee X_{i+1}, \ldots, X_n) & 0 < i < n \\
    (X_1,\ldots, X_{n-1}) & i = n
  \end{cases}
\end{equation*}
and the degeneracies wedge with a zero object. There is an inclusion
$N_\bullet \cat D \to S_\bullet \cat D$ that takes $(X_1,\ldots, X_n)$
to
\begin{equation*}
  \begin{tikzcd}
    * \ar[r, tail] & X_1 \ar[r, tail] \ar[d] & X_1 \vee X_2 \ar[d] \ar[r, tail, "\cdots" description]  & X_1\vee \cdots \vee X_n \ar[d]  \\
    & * \ar[r, tail] \ar[drr, "\ddots" description, phantom] & X_2 \ar[r, tail, "\cdots" description]  & X_2\vee \cdots \vee X_n \ar[d, "\vdots" description, phantom]\\ 
    & & * \ar[r, tail] & X_n\ar[d] \\
    & & & *
  \end{tikzcd}
\end{equation*}
But the assumption that sequences of ingressions split in $\cat D$
implies that this inclusion induces an equivalence
$w N_\bullet \cat D \to w S_\bullet \cat D$. Waldhausen's own argument
for this is \cite{1126}*{Proposition 1.8.7} and reduces to inductively
studying the map.
\begin{equation*}
  \begin{tikzcd}[row sep = tiny]
    w N_\bullet S_n \cat D \ar[r] & w N_\bullet S_{n-1} \cat D \x w
    N_\bullet \cat D \\
    (A_1 \cof \cdots \cof A_n) \ar[r, maps to] &
    ((A_1\cof \cdots \cof A_{n-1}), A_n/A_{n-1})
  \end{tikzcd}
\end{equation*}
This is an equivalence if for every object $X$ in $\cat D$, the map
$j:\cat D \to \cat D_X$, the category of ingressions under $X$, that
sends $A$ to $X\cof X \vee A$ induces an equivalence
$w N_\bullet \cat D \to w N_\bullet \cat D_X$. This is ensured by the
splitting hypothesis on $\cat D$, since any $X\cof A$ is equivalent
via a zig-zag to $X\cof X \vee A/X$ and the nerve $N_\bullet$
identifies $X \cof X$ to $X\cof X\vee A/X$ by the face corresponding
to $(X,X\vee A/X)$. We summarize Waldhausen's equivalence.
\begin{proposition}\label{NequalsS}
  If every ingression splits in $\cat D$, a Waldhausen $\oo$-category,
  then $w N_\bullet \cat D \simeq w S_\bullet \cat D$.
\end{proposition}

Let $[X]$ denote an equivalence class of objects in $\cat D$.
$\Aut(X)$ is a topological monoid with classifying space
$B\Aut(X)$. Since ingressions split in $\cat D$, we have a map
$\Aut(X) \to \Aut(X')$ for any ingression $X\cof X'$, taking $f:X\to
X$ to $f\vee \id$ from $X \vee X'/X \to X\vee X'/X$. We consider the
homotopy colimit $\hocolim_{[X]\in \cat D_\dagger} B\Aut(X)$, which we
will denote $\hocolim_{\cat D_\dagger} B\Aut(X)$. For example, when
$\cat D = \Proj(R)$ is the category of finitely-generated projective
$R$-modules, one has the cofinal collection of objects $\{R^n\}$ with
block-sum maps $\GL_n(R) = \Aut(R^n) \to \GL_{n+1}(R) = \Aut(R^{n+1})$
and $\hocolim_{\Proj(R)_\dagger} B \Aut(X) \simeq B\GL(R)$.

We form the group completion $(\hocolim_{\cat D} B\Aut(X))^+$ and
observe that
\begin{equation*}
  K_0(\cat D) \x (\hocolim_{\cat D_\dagger} B\Aut(X))^+
\end{equation*}
is equivalent to $\loops |w N_\bullet \cat D|$ following
\cite{segal}. We now have the desired description.
\begin{theorem}
  If $\cat D$ is a Waldhausen $\oo$-category where all ingressions
  split up to homotopy, then
  \begin{equation*}
    K(\cat C) \simeq K_0(\cat C) \x (\hocolim_{\cat D_\dagger} B\Aut(X))^+
  \end{equation*}
\end{theorem}
\begin{corollary}[``plus equals $Q$'' for weight structures]\label{plus-equals-Q}
  If $\cat C_{\heart w}$ is the heart of a weight structure on a
  stable $\oo$-category,
  \begin{equation*}
    K(\cat C_{\heart w}) \simeq K_0(\cat C) \x (\hocolim_{\cat C_{\heart w, \dagger}} B \Aut(X))^+
  \end{equation*}
\end{corollary}

\section{Examples of weight structures and
  applications}\label{examples}
In this section, we provide an overview of several examples of weight
structures. We produce applications of our main theorem and conjecture
about some future directions for research.

\subsection{The stable category and cellular truncation}

Let $\Sp^\fin$ denote the category of finite spectra. There is a
standard Postnikov $t$-structure on $h\Sp^\fin$ where
$\Sp^\fin_{t\geq 0}$ contains all connective ($(-1)$-connected)
spectra and $\Sp^\fin_{t\leq 0}$ consists of all $1$-coconnective
spectra (\ie those with homotopy groups concentrated in degrees
$\leq 0$). The $t$-structure decompositions are provided by taking
$n$-connected covers and truncating homotopy groups at degree
$n-1$. The heart of this $t$-structure is the abelian category of
finitely-generated groups included as the Eilenberg--Maclane spectra.

The weight structure is generated by the sphere spectrum. Let $B$
denote the collection of finite wedges of the sphere spectrum
$S^0$. $\Sp_{w\geq 0}$ is defined to be those spectra $E$ with
$h\Sp(X,E) = 0$ for any $X$ in $\Sigma^{k}B$ for $k< 0$. In other
words, $\Sp_{w\geq 0}$ is the subcategory of connective
spectra. $\Sp_{w\leq 0}$ is the defined via the orthogonality
condition: $E \in \Sp_{w\leq 0}$ if $h\Sp(E,Y)=0$ for all
$Y\in \Sp_{w\geq 1}$. $\Sp_{w\leq 0}$ turns out to contain those
spectra which are $0$-skeleta---which can be distinguished by their
homology: $E \in \Sp_{w\leq n}$ if and only if $H\Z_*(E) = 0$ for
$* > 0$ and $H\Z_0(E)$ is a free abelian group. The heart
$\Sp_{\heart w}$ of this weight structure consists of all spectra
weakly equivalent to finite wedges of $S^0$.

\begin{corollary}
  The $K$-theory of the sphere spectrum, $K(S^0)\coloneq K(\Sp_\fin)$,
  is equivalent to the $K$-theory of the full additive
  $\oo$-subcategory on finite wedge sums of the sphere spectrum
  $S^0$. The latter has a plus-construction description $K(S^0) \simeq
  K(\Z) \x B\underline{\Aut}(S^0)^+$, where $\underline{\Aut}(S^0)$ is
  the homotopy colimit $\hocolim_{n \in (\textrm{FinSet},\textrm{Inj})}
  B\Aut(\bigvee_n S^0, \bigvee_n S^0)$ over the skeleton of the category of
  finite sets and injections.
\end{corollary}

The cellular truncation applies to many other cases where there is
already a robust notion of cellularity. For example, when $R$ is a
connective ring spectrum, every finite cell spectrum over $R$ admits a
CW-cell structure. Hence, there is a weight structure on $f\cat C_R$,
the $\oo$-category of finite cell $R$-modules (in the sense of
\cite{EKMM}), where weight decompositions are given by truncating with
a CW-skeleton. Likewise, the $\oo$-category $f\cat{CW}_R$ of finite
CW-spectra over $R$ admits a cellular truncation weight structure as
well. The heart of both weight structures will consist of retracts of
finite wedge sums of copies of $R$.
\begin{corollary}[see also \cite{EKMM}*{IV.3.1}]
  If $R$ is a connective ring spectrum, then the $K$-theory of finite
  CW-spectra over $R$ and finite cell spectra over $R$ are equivalent
  $K(f\cat{CW}_R) \simeq K(f\cat{C}_R)$ and are both equivalent to the
  $K$-theory of the full additive $\oo$-subcategory of retracts of finite
  wedge sums of $R$.
\end{corollary}
The plus-construction description of $K(R)$ is equivalent to that of
\cite{EKMM}*{IV.7.1}.

\subsection{Chain complexes, the brutal truncation, and a theorem of
  Gillet and Waldhausen}

Let $\cat E$ be an exact category (in the sense of Quillen). We
consider the Waldhausen category $\Ch^\bdd(\cat E)$ of bounded chain
complexes in $\cat E$. Cofibrations in this category are admissible
monomorphisms and weak equivalences are chain homotopy
equivalences. The relative nerve produces a corresponding stable
Waldhausen $\oo$-category which also denote $\Ch^\bdd(\cat A)$ by
abuse of notation. We defend this notation with the fact that the
algebraic $K$-theory spectra of these objects coincide
\cite{barwick}*{10.16}.

$\Ch^\bdd(\cat E)$ admits a bounded weight structure where
$\Ch^\bdd(\cat E)_{w\geq 0}$ consists of chain complexes (chain)
homotopy equivalent to ones concentrated in nonnegative
degrees. Similary, $\Ch^{\bdd}(\cat E)_{w\leq 0}$ consists of
complexes homotopy equivalent to complexes concentrated in nonpositive
degrees. The heart of this weight structure is chain complexes
homotopy equivalent to complexes concentrated in degree $0$, hence it
is equivalent to $\cat E$ where morphisms between objects are replaced
by morphisms in $\Ch^{\bdd(\cat E)}$. Weight decompositions are
provided by the \term{brutal truncation}:
\begin{equation*}
  \begin{tikzcd}
   M_{w\leq n}  \ar[d]&   \Big(\cdots & \ar[l] M_{n-1} \ar[d,equals] & \ar[l] M_{n} \ar[d, equals] & \ar[l] 0 \ar[d] & \ar[l] 0 \ar[d] & \ar[l] \cdots\Big) \\
   M\ar[d] & \Big(\cdots & \ar[l] M_{n-1} \ar[d] & \ar[l] M_{n} \ar[d] & \ar[l] M_{n+1} \ar[d, equals] & \ar[l] M_{n+2} \ar[d, equals] & \ar[l] \cdots\Big) \\
    M_{w\geq n+1} & \Big(\cdots & \ar[l] 0 & \ar[l] 0 & \ar[l] M_{n+1} & \ar[l] M_{n+2} &
    \ar[l] \cdots\Big)
  \end{tikzcd}
\end{equation*}
If we write $\Ch^\bdd(\cat E)_{\heart w}$ for the heart of this weight
structure, our main theorem implies the following.
\begin{corollary}\label{cor_K_of_E}
  The algebraic $K$-theory of $\Ch^\bdd(\cat E)$ is equivalent to
  Quillen's algebraic $K$-theory of $\cat E$.
\end{corollary}
\begin{proof}
  The main theorem implies that $K(\Ch^\bdd(\cat E)_{\heart w}) \simeq
  K(\Ch^\bdd(\cat E))$ since the brutal truncation provides a bounded
  weight structure on bounded complexes. The heart $\Ch^\bdd(\cat
  E)_{\heart w}$ is equivalent to the simplicial nerve of $\cat E$
  (morphisms between objects in $\cat E$ are given by simplicial
  mapping spaces of morphisms between their chain complexes). By
  \cite{2013arXiv1301.4725B}*{3.11}, the Barwick $K$-theory of this
  (exact) $\oo$-category coincides with that from Quillen's
  $Q$-construction.
\end{proof}

\begin{lemma}\label{lemma_qi_chhtpy}
  Let $\cat E$ be an exact category which is idempotent-complete and
  let $\Ch^\bdd(\cat E)$ denote the Waldhausen category of bounded
  chain complexes on $E$ with cofibrations levelwise admissible
  monomorphisms. The $K$-theory spectrum $K(\Ch^\bdd(\cat E))$ is
  equivalent whether chain homotopy equivalences of quasi-isomorphisms
  are taken as weak equivalences.
\end{lemma}
\begin{proof}
  Let $v_\chhtpy$ and $v_\qi$ denote the class of chain homotopy
  equivalences and quasi-isomorphisms, respectively. Consider the
  localization sequence
  \begin{equation*}
    K(\Ch^\bdd_\ac(\cat E), v_\chhtpy) \to K(\Ch^\bdd(\cat E),v_\chhtpy) \to K(\Ch^\bdd(\cat E),v_\qi)
  \end{equation*}
  where $\Ch^\bdd_\ac(\cat E)$ denotes the full Waldhausen subcategory
  on acyclic bounded complexes. The lemma follows from proving the
  $K$-theory of the bounded acyclics is trivial.

  It suffices to check for non-negative chain complexes. Filter the
  acyclics by subcategories $\Ch^{\bdd,\leq n}_\ac(\cat E)$ of acyclic
  complexes concentrated in degrees $0 \leq * \leq n$. The short exact
  sequence of complexes
  \begin{equation*}
    \begin{tikzcd}
      0 \ar[d] \ar[r] & 0 \ar[d] \ar[r] & 0 \ar[d] \\
      M_n \ar[d, "\isom"] \ar[r, equals] & M_n \ar[r] \ar[d, "d_n"] & 0 \ar[d] \\
      \im(d_n) \ar[d] \ar[r, hook] & M_{n-1} \ar[d, "d_{n-1}"] \ar[r, two heads] & \coker(d_n) \ar[d, "d_{n-1}"] \\
      0 \ar[r]\ar[d] & M_{n-2} \ar[r, equals] \ar[d] & M_{n-2} \ar[d] \\ 
      \vdots & \vdots & \vdots
    \end{tikzcd}
  \end{equation*}
  gives a short exact sequence of endofunctors on the acyclics
  $\Ch^{\bdd,\leq n}_\ac(\cat E)$. We note that $\cat E$ is required
  to be closed under idempotents for these to take values in chain
  complexes on $\cat E$. If $\cat E$ is not idempotent-complete, one
  can pass to it's idempotent closure or Karoubi envelope to apply
  this result.

  The first column is an elementary acyclic complex composed of single
  isomorphism. Elementary acyclic complexes are all chain homotopy
  equivalent to the zero complex. Hence, the algebraic $K$-theory of
  the elementary acyclics is trivial. Since the right column takes
  values in $\Ch_\ac^{\bdd, \leq n-1}(\cat E)$, we can apply the
  additivity theorem to conclude $K(\Ch_\ac^{\bdd, \leq n}(\cat E))
  \simeq K(\Ch_\ac^{\bdd, \leq n-1}(\cat E))$. Since $\Ch_\ac^{\bdd,
    \leq 1}(\cat E)$ is equivalent to the category of elementary
  acyclics concentrated in degrees $0$ and $1$, we conclude
  $K(\Ch^{\bdd, \leq n}_\ac(\cat E))\simeq *$ for all $n$ and conclude
  that $K(\Ch^\bdd_\ac(\cat E))\simeq *$ since $K$-theory commutes
  with directed colimits.
\end{proof}

We now arrive at a new proof of the Gillet--Waldhausen Theorem.
\begin{corollary}[Gillet--Waldhausen Theorem,\cite{MR1106918}*{1.11.7}]\label{gillet-wald}
  For an exact category $\cat E$ which is idempotent-complete,
  Quillen's algebraic $K$-theory $K(\cat E)$ is homotopy equivalent to
  the Waldhausen algebraic $K$-theory of $\Ch^\bdd(\cat E)$, the
  Waldhausen category of bounded chain complexes on $\cat E$ where
  cofibrations are taken to be admissible monomorphisms and weak
  equivalences are quasi-isomorphisms of chain complexes.
\end{corollary}
\begin{proof}
  Corollary \ref{cor_K_of_E} implies that Quillen's $K(\cat E)$ is
  equivalent to $K(\Ch^\bdd(\cat E), v_{\chhtpy})$ and Lemma
  \ref{lemma_qi_chhtpy} implies that $K(\Ch^\bdd(\cat E), v_\chhtpy)
  \simeq K(\Ch^\bdd(\cat E), v_\qi)$ as desired.
\end{proof}
\begin{corollary}
  If $\Proj^\fg(R)$ denotes the exact category of finitely-generated projective
  $R$-modules, $K(\Proj^\fg(R))\simeq K(\Ch^\bdd(\Proj^\fg(R)))$.
\end{corollary}

\subsection{The Resolution Theorem for exact categories and $G=K$ by
  way of the brutal truncation}

For a ring $R$, the $G$-theory of $R$ is defined to be algebraic
$K$-theory of the exact category of finitely-generated modules over
$R$.

\begin{corollary}[Resolution Theorem, \cite{MR3076731}*{Theorem
    V.3.1}]\label{resolution_theorem}
  Let $\cat P$ be a full subcategory of an exact category $\cat E$ so
  that $\cat P$ is closed under extensions and under kernels of
  admissible surjections in $\cat E$. Suppose in addition that every
  object $M$ in $\cat E$ admits a finite $\cat P$-resolution:
  \begin{equation*}
    0 \to P_n\to P_{n-1} \to \cdots \to P_1 \to P_0 \to M \to 0
  \end{equation*}
  then $K(\cat P)\simeq K(\cat E)$.
\end{corollary}
\begin{proof}
  The hypotheses imply that every bounded chain complex in $\cat E$
  admits a bounded quasi-isomorphic $\cat P$-replacement. Put a weight
  structure on $\Ch^\bdd(\cat E)$ by fixing $\cat P$-replacements and
  applying brutal truncation. The heart of this weight structure is
  equivalent to $\cat P$, included as complexes concentrated in degree
  $0$. Applying the sphere theorem, we get that $K(\cat P)\simeq
  K(\Ch^\bdd(\cat E))$. The standard brutal truncation on
  $\Ch^\bdd(\cat E)$ produces an equivalence $K(\Ch^\bdd(\cat
  E))\simeq K(\cat E)$ by Corollary \ref{cor_K_of_E}.
\end{proof}
For Noetherian regular rings, all modules admit finite projective resolutions. The standard corollary of the resolution theorem follows.
\begin{corollary}[``$G=K$'']
  When $R$ is a Noetherian regular ring, $K(R) \simeq G(R)$.
\end{corollary}
\begin{proof}
  The resolution theorem implies $K(\Proj^\fg(R))\simeq
  K(\Mod^\fg(R))$.
\end{proof}

\subsection{Categories of motives}

In \cite{2009arXiv0903.0091B}, Bondarko establishes a weight structure
on Voevodsky's triangulated category of motives. These are constructed
from the bounded chain complexes of smooth varieties (with smooth
correspondences as morphisms) by localizing and forming the idempotent
completion. Within this category, the \emph{Chow motives} are cut out
by smooth projective varieties (with morphisms smooth correspondences
modulo rational equivalences). Bondarko builds a ``Chow'' weight
structure on $\operatorname{DM}^{\textrm{eff}}_{\textrm{gm}}$ the
category of effective geometric motives whose heart is the effective
Chow motives \cite{MR2746283}. His $K_0$ version of the sphere theorem
computes $K_0(\operatorname{DM}^{\textrm{eff}}_{\textrm{gm}}) \simeq
K_0(\operatorname{Chow}^{\textrm{eff}})$.

Effective geometric motives are constructed as a localization of
presheaves of abelian groups on smooth schemes. Let
$\operatorname{DM}^{\textrm{eff}}_{\textrm{gm}, \oo}$ denote the stable
$\oo$-category produced by this construction, so that the triangulated
category of effective geometric motives is its homotopy
category. Bondarko's weight structure produces a weight structure on
$\operatorname{DM}^{\textrm{eff}}_{\textrm{gm},\oo}$ whose heart is the full
additive $\oo$-subcategory on the effective Chow motives, which we
denote $\operatorname{Chow}^{\textrm{eff}}_\oo$. Applying the sphere
theorem gives the following new result.
\begin{corollary}
  There is an equivalence of connective $K$-theory spectra
  \begin{equation*}
    K(\operatorname{DM}^{\textrm{eff}}_{\textrm{gm},\oo}) \simeq
    K(\operatorname{Chow}^{\textrm{eff}}_\oo).
  \end{equation*}
\end{corollary}
On $\pi_0$, this result reproduces that of Bondarko.

\subsection{Conjectural weight structures}

Blumberg, Gepner, and Tabuada introduce a category
$\mathcal{M}_{\textrm{loc}}$ of ``localizing noncommutative (spectral)
motives'' in \cite{BGT}.  This category is constructed from the
category of spectrum-valued presheaves on the $\oo$-category of small
stable $\oo$-categories. This is the category where non-connective
$K$-theory is co-represented. Blumberg has conjectured that
$\mathcal{M}_{\textrm{loc}}$ admits a weight structure whose heart is
the dualizable objects---the smooth and proper $dg$-categories.

Following Hill--Hopkins--Ravenel, the category of genuine $G$-spectra
admits interesting ``slice filtrations''. These are equivariant
analogues for Postnikov towers. The author conjectures that there are
adjacent slice weight structures generated by wedge sums of the
regular representation spheres. The heart of this weight structure
would contain finite wedge sums of all finite-dimensional
representation spheres concentrated in virtual degree
$0$.


\bibliographystyle{amsalpha}
\bibliography{references}

\end{document}